\documentclass{amsart}
\usepackage{amsfonts,amsmath}

\usepackage{color,amssymb,sidenotes}
\usepackage{tikz-cd,amscd,amsmath}
\usepackage[all,cmtip]{xy}

\newtheorem{Theorem}{\bf Theorem}
\newtheorem{lemma}[Theorem]{\bf Lemma}
\newtheorem{proposition}[Theorem]{\bf Proposition}
\newtheorem{corollary}[Theorem]{\bf Corollary}

\newtheorem{definition}[Theorem]{\bf Definition}

\newtheorem{remark}[Theorem]{\bf Remark}
\newtheorem{theorem}[Theorem]{\bf Theorem}

\newtheorem*{theorem*}{\bf Theorem}
\newtheorem*{example*}{\bf Example} 

\usepackage{xcolor}

\begingroup\lccode`~=`& \lowercase{\endgroup
\newcommand{\spmat}[1]{
  \left[
  \let~=&
  \begin{smallmatrix}#1\end{smallmatrix}
  \right]
}}

\newcommand{\m}{\mathfrak{m}} 
\newcommand{\N}{\mathbb{N}}

\newcommand{\n}{n}

\newcommand{\LL}{\mathcal{L}}
\newcommand{\RR}{\mathcal{R}}
\newcommand{\s}{\mathcal{S}}

\newcommand{\ov}{\overline}
\newcommand{\lra}{\longrightarrow}

     \title[Joint reductions and mixed Buchsbaum-Rim multiplicities of modules]{Joint reductions and mixed Buchsbaum-Rim multiplicities of modules and a joint-reduction-number-zero theorem}
     \author{Daniel Katz}
     \address{Department of Mathematics, University of Kansas, Lawrence, Kansas 66045, USA}
     \email{dlk53@ku.edu}
      \author{Vijay Kodiyalam}
     \address{The Institute of Mathematical Sciences, Chennai, India and Homi Bhabha National Institute, Mumbai, India}
     \email{vijay@imsc.res.in}
     \author{J. K. Verma}
     \address{Indian Institute of Technology Gandhinagar, Palaj, Gandhinagar 382355, Gujarat, India}
     \email{jugal.verma@iitgn.ac.in}

\date{February 28, 2025}
\dedicatory{Dedicated to the memory of J\"{u}rgen Herzog whose pioneering work in commutative algebra was an inspiration to generations of algebraists.}

\begin{document}
     \begin{abstract}
     We offer new definitions of joint reductions and mixed Buchsbaum-Rim multiplicity for certain collections of modules over a Noetherian local ring  and illustrate their application to give two different proofs of a joint-reduction-number-zero theorem for integrally closed modules over two-dimensional regular local rings. We also relate the mixed Buchsbaum-Rim multiplicity of modules to the Euler-Poincar\'{e} characteristic of a natural Koszul complex and relate it to the mixed Buchsbaum-Rim multiplicity of ideals by generalising a lemma from intersection theory.
       \end{abstract}
     \maketitle

 In this paper we  study the notions of joint reductions and  mixed Buchsbaum-Rim multiplicity for certain collections of modules over Noetherian local rings and apply these ideas to a family of integrally closed modules in a two-dimensional regular local ring. We were initially motivated by a question posed to the third author by S.~Kleiman. The question asks whether the following theorem of Rees (see \cite{Res1981}) might extend to a pair of finitely generated modules of finite colength contained in a corresponding pair of free modules. In the theorem below, $\ov{e}_2(I)$ refers to the constant term of the degree two polynomial giving the lengths of $R/\overline{I^rJ^s}$, and $\overline{I^rJ^s}$ denotes the integral closure of the ideal $I^rJ^s$.
 
 \begin{theorem}{\bf (Rees)} \label{Rees} Let $I$ and $J$ be $\m$-primary ideals of a two-dimensional analytically unramified Cohen-Macaulay local ring with infinite residue field. Then the following are equivalent:\\
(1) $\ov{e}_2(IJ)=\ov{e}_2(I)+\ov{e}_2(J).$\\
(2) There exist $a\in I$ and $b\in J$ such that for all  $r, s\geq 1,$ 
$$\ov{I^rJ^s}=a\ov{I^{r-1}J^s}+b\ov{I^rJ^{s-1}}.$$
Moreover each of these conditions implies that $\ov{IJ}=\ov{I}\cdot \ov{J}.$
\end{theorem} 
 
Note that in a two-dimensional pseudo-rational or regular local ring, where products of integrally closed ideals are integrally closed, for two integrally closed ideals $I$ and $J$, the terms in (1) are zero and (2) states that the joint reduction number of the pair of ideals is zero (see \cite{Vrm1990}).  By the work of the second author (\cite{Kdy1995}), if $R$ is a two-dimensional regular local ring, a product of integrally closed modules is integrally closed, and by the work of the first and second author (\cite{KtzKdy1997}), the polynomial giving the lengths of $S_n(F)/S_n(M)$ just has two terms (when expressed as a combination of binomial coefficients), where $M$ is an integrally closed module of finite colength in the free module $F$, $S_n(F)$ is the $n$th symmetric power of $F$ and $S_n(M)$ is the image in $S_n(F)$ of the $n$th symmetric power of $M$. These results suggest that conditions (1) and (2) hold for integrally closed modules when $R$ is a two-dimensional regular local ring, so that Kleiman's question would have a (strong) positive answer in this case. We are able to achieve this goal. In particular, in sections two and five we offer two different proofs of a joint-reduction-number-zero condition for integrally closed modules over a two-dimensional regular local ring. In section six, we give an explicit form for the joint Buchsbaum-Rim polynomial for several integrally closed modules that specialises to give a formula for the length of $S_{n_1}(F_1)S_{n_2}(F_2)/S_{n_1}(M_1)S_{n_2}(M_2)$, for all $n_1, n_2\geq 0$. 
 
 Our work on the problem of considering two (or more) integrally closed modules required us to revisit the concepts of joint reductions for modules and the resulting mixed Buchsbaum-Rim multiplicity. 
 Earlier definitions of joint reductions exist in the work of several authors, including \cite{KrbRes1994} and \cite{CllPrz2010}, but in this paper we offer a different definition of a joint reduction for a family of modules which we feel generalises the ideal case in a natural way.  Similarly, mixed Buchsbaum-Rim multiplicities of modules appear in \cite{KrbRes1994}, \cite{CllPrz2010}, and \cite{KlmThr1996} (see also \cite{Cdr2024}), but our take on these multiplicities is not the same as the one offered in \cite{CllPrz2010}; is less general, but more precise (for our applications) than the approach in \cite{KrbRes1994}; and is purely algebraic, and so differs from the view presented in \cite{KlmThr1996}.\footnote{For example, \cite{CllPrz2010} considers a finite family of modules contained in a single free module, and takes one element from each module to construct a joint reduction. The joint Buchsbam-Rim polynomial giving the mixed Buchsbaum-Rim multiplicities in \cite{CllPrz2010} has degree $d+r-1$, just as in the case of the Buchsbaum-Rim polynomial for one module, where $d$ is the dimension of the ring and $r$ is the rank of the module. The mixed Buchsbaum-Rim multiplicity we work with comes from a polynomial of degree $r_1+\cdots + r_d$, where we have taken $d$ modules of finite colength in $d$ free-modules of ranks $r_1, \ldots, r_d$. Kirby and Rees work in the general setting of multigraded algebras, and the mixed Bucshbaus-Rim multiplicities in \cite{KlmThr1996} are intersection numbers.}
 
To indicate our approach to joint reductions and the mixed Buchsbaum-Rim multiplicity, first consider the (most natural) case of $d$ $\m$-primary ideals $I_1, \ldots, I_d$ in a local ring $(R, \m, k)$ with dimension $d$ and infinite residue field $k$, introduced by Rees in \cite{Res1984}. 
Elements $x_1\in I_1, \ldots, x_d\in I_d$ form a joint reduction of $I_1, \ldots, I_d$ if for some $n \in {\mathbb N}_0 = {\mathbb N} \cup \{0\}$, 
\begin{align*}
I_1^{n+1}I_2^{n+1}\cdots I_d^{n+1} = x_1I_1^{n}I_2^{n+1}\cdots I_d^{n+1} + x_2I_1^{n+1}I_2^n \cdots I_d^{n+1} \\
+ \cdots + x_dI_1^{n+1}\cdots I_{d-1}^{n+1}I_d^n.
\end{align*}
 {\it The} (emphasis intended) mixed multiplicity $e(I_1|\cdots |I_d)$ of the ideals $I_1, \ldots, I_d$ is the normalised coefficient of the $n_1\cdots n_d$ term appearing in the multivariable Hilbert-Samuel polynomial giving the length of $R/I_1^{n_1}\cdots I_d ^{n_d}$, for all $n_i$ sufficiently large. Rees shows that $e(I_1|\cdots |I_d)$ is the multiplicity of the ideal $x_1, \ldots, x_d$, which in turn is the Euler-Poincar\'{e} characteristic of the Koszul complex on $x_1, \ldots, x_d$. Now suppose that $F_1, \ldots, F_d $ are free $R$-modules of ranks $r_1, \ldots, r_d$ and $M_1\subseteq F_1, \ldots, M_d\subseteq F_d$ are modules of finite colength. Our view is that an $r_i$-generated submodule $B_i\subseteq M_i$ is to $M_i$ as a single element $x_i\in I_i$ is to $I_i$. 
Then, the definition of $B_1,\cdots,B_d$ being a joint reduction of $M_1,\cdots,M_d$ is a natural equational generalisation of the definition for ideals.
In this setting, the joint Buchsbaum-Rim polynomial $p(n_1, \ldots, n_d)$ associated to $M_1, \ldots, M_d$ has total degree $r_1+\cdots + r_d$ (see  \S7) and the mixed Buchsbaum-Rim multiplicity $br(M_1|\cdots | M_d)$ is the normalised coefficient of the $n_1^{r_1}\cdots n_d^{r_d}$ term. To calculate $br (M_1|\cdots |M_d)$, we use the Koszul-like complex $K(\phi_1)\otimes \cdots \otimes K(\phi_d)$, where $K(\phi_i)$ is the complex $0\to F_i\overset{\phi_i}\to F_i\to 0$ and $\phi_i$ is the map whose matrix with respect to the standard basis of $F_i$ is $B_i$. Using the Euler-Poincar\'{e} characteristic of this complex, it follows from sections three and four that
\[
e(det(\phi_1), \ldots, det(\phi_d)) = br (M_1|\cdots | M_d) = e(I_1|\cdots | I_d),
\] where $I_i$ is the ideal of $r_i\times r_i$ minors of the matrix whose columns are the generators of $M_i$. 
 
Maintaining the notation above, here is a brief description of the contents of this paper.  Throughout $(R, \m, k)$ is a local ring of dimension $d > 0$ with infinite residue field $k$. Given a bounded  complex $\mathcal{C}$ of $R$-modules with finite length homology modules, we will write $\chi(\mathcal{C})$ for the Euler-Poincar\'{e} characteristic of $\mathcal{C}$, i.e., the alternating sum of the lengths of the homology modules of $\mathcal{C}$. Note that if the modules in $\mathcal{C}$ have finite length, then $\chi(\mathcal{C})$ also equals the alternating sum of the lengths of the modules in $\mathcal{C}$. 
 In \S1 we define joint reductions of the family of modules $M_1, \ldots, M_q$ over Noetherian local rings of positive dimension. We show that with our definition, joint reductions exist and also prove some basic properties of joint reductions. In particular, we show the equivalence of the equational, valuative and determinantal definitions and note that if the ring $R$ has positive depth, then the joint reduction $B_1, \ldots, B_q$ of $M_1, \ldots, M_q$ can be chosen so that each $B_i$ is a free $R$-module. We apply this in \S2 to give a proof of the joint-reduction-number-zero theorem for integrally closed modules over a two-dimensional regular local ring.
 In \S3 we return to the setting of a general Noetherian local ring of positive dimension $d$ and define the mixed Buchsbaum-Rim multiplicity of a collection of $d$ modules of finite colength in free modules. We also relate this to the Euler-Poincar\'{e} characteristic of a naturally defined Koszul complex. In \S4 we prove a generalisation of a lemma from intersection theory and use it to relate the mixed Buchsbaum-Rim multiplicity of modules to that of the ideals of maximal minors of these modules. In \S5 we get back to the context of two-dimensional regular local rings and give a different proof of the  joint-reduction-number-zero theorem for integrally closed modules using the Hoskin-Deligne length formula for modules and the results of the previous two sections. The penultimate \S6 applies the joint-reduction-number-zero result to explicitly calculate the joint Buchsbaum-Rim function for a finite collection of integrally closed modules over a two-dimensional regular local ring. The final \S7 is an appendix where we give a self-contained proof of the existence of the joint Buchsbaum-Rim polynomial for a collection of modules of finite colength in free modules over general local rings.

 \section{Joint reductions of modules}
 
Let $(R,{\mathfrak m},k)$ be a Noetherian  local ring of  positive  dimension $d$ with infinite residue field. 
Assume that $M_1 \subseteq F_1,M_2 \subseteq F_2, \cdots,$ $ M_q \subseteq F_q$ are $R$-submodules  of  finite colength in free   $R$-modules $F_1,\cdots,F_q$ of ranks $r_1,\cdots,r_q$.  In this section, we will provide a new definition for a joint reduction of the collection $M_1, \ldots, M_q$ and show that joint reductions exist if $q\geq d$. We also derive some properties of joint reductions, including: alternate characterisations of joint reductions; minimal generating sets for elements in the joint reduction extend to minimal generating sets of $M_1, \ldots, M_q$; elements in the joint reduction can be chosen to be free submodules of $M_1, \ldots, M_q$, if $R$ has positive depth. These are properties that hold in the special case of each $M_k$ being an ideal of $R$. 
In the definition below, we will allow for the possibility that some (or all) $M_k=F_k$.  
Set $F = F_1 \oplus F_2 \oplus \cdots \oplus F_q$ and $M = M_1 \oplus M_2 \oplus \cdots \oplus M_q$. Then $S(M) \subseteq S(F)$ is an inclusion of ${\mathbb N}^q$-graded $R$-algebras, where $S(M)$ denotes the image of the symmetric algebra of $M$ in the symmetric algebra of $F$. Suppose that $B_1 \subseteq M_1,B_2 \subseteq M_2,\cdots,B_q \subseteq M_q$ are $R$-submodules.
     
\begin{definition}\label{def:jtred}
The collection $(B_1,\cdots,B_q)$ is said to be a joint reduction of the collection $(M_1,\cdots,M_q)$ if
\begin{itemize}
\item Each $B_k$ is $r_k$-generated, and 
\item For some $n \in {\mathbb N}_0 = {\mathbb N} \cup \{0\}$, the joint reduction equation
\begin{align*}S_{n+1}(M_1)S_{n+1}(M_2)\cdots S_{n+1}(M_q) &= B_1S_{n}(M_1)S_{n+1}(M_2)\cdots S_{n+1}(M_q)\\
 &+ S_{n+1}(M_1)B_2S_{n}(M_2)\cdots S_{n+1}(M_q)\\
 &+\cdots\\
 &+S_{n+1}(M_1)S_{n+1}(M_2)\cdots B_qS_{n}(M_q),
\end{align*}
holds inside $S(M)$, or equivalently, if the ideal generated by all of $B_1,\cdots,B_q$ in $S(M)$ is irrelevant in the sense that it contains a power of the ideal generated by
$S(M)_{(1,1,\cdots,1)}$.
\end{itemize}
If $(B_1,\cdots,B_q)$ is a joint reduction of  $(M_1,\cdots,M_q)$, the smallest $n$ for which the joint reduction equation holds is said to be the joint reduction number of 
 $(M_1,\cdots,M_q)$ with respect to $(B_1,\cdots,B_q)$.
\end{definition}   

We first show that if $q \geq d$, then  joint reductions in the sense of Definition~\ref{def:jtred} always exist as a consequence of the results of Kirby and Rees in \cite{KrbRes1994}. Note that if $M_k=F_k$ then we may take $B_k=F_k$ and arbitrary $r_j$-generated submodules $B_j \subseteq M_j$, for $j \neq k$, to get a joint reduction, so the proof of existence is possibly non-trivial only in the case all the $M_k$ are proper submodules of $F_k$.
     
\begin{proposition}\label{prop:exist}
With notation as above, if $q \geq d$, there exists a joint reduction   $(B_1,\cdots,B_q)$  of  $(M_1,\cdots,M_q)$. More precisely, identfying $M_k$ with the $r_k \times m_k$ matrix of its generators,
there is a non-empty Zariski
open subset $U$ of \[M_{m_1 \times r_1}(k) \times \cdots \times M_{m_q \times r_q}(k)\] such that if $(\overline{A}_1,\cdots,\overline{A}_q) \in U$ (with $A_k \in M_{m_k \times r_k}(R)$ and $\overline{A}_k \in M_{m_k \times r_k}(k)$ denoting its reduction modulo ${\mathfrak m}$) and $B_k = M_kA_k$, then
$(B_1,\cdots,B_q)$  is a joint reduction of  $(M_1,\cdots,M_q)$.

\end{proposition}
     
Before proceeding with the proof, we briefly recall the contents of Section~1 of \cite{KrbRes1994} in their own notation. Their main object of interest is an ${\mathbb N}^q$-graded algebra $G$ with $G_{(0,0,\cdots,0)} = Q$ - a Noetherian local ring with maximal ideal ${\mathfrak m}$ and infinite residue field $k$. The algebra $G$ is also assumed to be generated by all the $G_{\delta_i}$ - where $\delta_i \in {\mathbb N}^q$ is the vector with a single non-zero entry 1 in the $i^{th}$ position - which are assumed to be finitely generated $Q$-modules. A relevant ideal of $G$ is one that does not contain a power of $G_{(1,1,\cdots,1)}$.

The dimension of $G$ (which is different from the Krull dimension of $G$) is defined to be $-1$ if the zero ideal of $G$ is irrelevant (i.e., $G_{(1,1,\cdots,1)}$ generates a nilpotent ideal) and to be the maximum of the heights of relevant ${\mathbb N}^q$-graded  prime ideals of $G$ otherwise. The analytic spread of $G$, which we denote by $a(G)$, is defined to be one more than the dimension of $G/{\mathfrak m}G$. Both of these can be computed by ``restricting to the diagonal" and thereby dealing with just ${\mathbb N}$-graded rings. Specifically, if $\Delta G$ denotes the ${\mathbb N}$-graded subring of $G$ generated by $G_{(1,1,\cdots,1)}$, then, Lemma 1.7 of \cite{KrbRes1994} shows that the dimension and analytic spread of $G$ coincide with those of $\Delta G$. 
 
For $a = (a_1,\cdots,a_q) \in {\mathbb N}^q$, a joint reduction of $G$ of type $a$ is a collection of elements \[\{z(k,j) \in G_{\delta_k}: k=1,2,\cdots,q ~\text{and}~ j = 1,2,\cdots,a_k\}\] generating an irrelevant ideal of $G$. The main result - Theorem 1.6 - of Section~1 of \cite{KrbRes1994} shows that if $|a| =  a_1+\cdots+a_q$ is at least the analytic spread of $G$, then there exists a joint reduction of $G$ of type $a$.
In this paper, we will need to use more than just the statement of this theorem. Suppose that for $k = 1,2,\cdots,q$, the $Q$-module $G_{\delta_k}$ is generated  by $y(k,l)$ for $l=1,\cdots,m_k$. Consider matrices $A_1,\cdots,A_q$ over $Q$ of sizes $m_1 \times a_1,\cdots,m_q \times a_q$ respectively. Define $z(k,j) \in G_{\delta_k}$ by the matrix equalities
$$
\left[
\begin{array}{ccc}
z(k,1) & \cdots & z(k,a_k)
\end{array}
\right]
=
\left[
\begin{array}{ccc}
y(k,1) & \cdots & y(k,m_k)
\end{array}
\right]~A_k.
$$
The proof of Theorem 1.6  of \cite{KrbRes1994} shows that there is a non-empty Zariski open subset $U$ of $M_{m_1 \times a_1}(k) \times \cdots \times M_{m_q \times a_q}(k)$ such that if $(\overline{A}_1,\cdots,\overline{A}_q) \in U$, then, the collection \[\{z(k,j): k=1,2,\cdots,q ~\text{and}~ j = 1,2,\cdots,a_k\}\] forms a joint reduction of $G$ of type $a$, provided $|a|$ is at least the analytic spread of $G$. In other words, the collection of $a_k$ `general linear combinations' of the generators of $G_{\delta_k}$ over all $k=1,\cdots,q$ forms a joint reduction of $G$.

Note that for an ${\mathbb N}$-graded ring $G$ with $G_0 = k$ - a field - and satisfying $G = G_0[G_1]$ with $G_1$ being a finite-dimensional $k$-vector space,  the dimension of $G$  is the degree of the polynomial that gives $dim_k(G_n)$ for all large $n$, which is one less than the Krull dimension of $G$ (see\cite{BrnHrz1993}, Theorem 17.16).

\begin{proof}[Proof of Proposition \ref{prop:exist}] As observed earlier, we may assume that all $M_i$ are proper submodules of $F_i$. We will then apply the results of \cite{KrbRes1994} to $G=S(M)$ and $Q=R$  with $S(M)$ being an ${\mathbb N}^q$-graded $R$-algebra with $$S(M)_{(n_1,\cdots,n_q)} = S_{n_1}(M_1)S_{n_2}(M_2) \cdots S_{n_q}(M_q).$$ With $r = (r_1,\cdots,r_q) \in {\mathbb N}^q$, a joint reduction of $S(M)$ of type $r$ is guaranteed to exist provided $|r| = r_1 + \cdots + r_q$ is at least the analytic spread of $S(M)$, which is true by Proposition \ref{prop:ansp} when $q \geq d$.

Thus for each $k=1,2,\cdots,q$, there exist $r_k$ elements $z(k,j) \in S(M)_{\delta_k} = M_k$ for $j=1,2,\cdots,r_k$, such that the ideal generated by all the $z(k,j)$ in $S(M)$ is irrelevant. Let $B_k$ denote the submodule of $M_k$ generated by all the $z(k,j)$ for $j=1,2,\cdots,r_k$. Then $B_k$ is an $r_k$-generated submodule of $M_k$ and the collection  $(B_1,\cdots,B_q)$  is a joint reduction of $(M_1,\cdots,M_q)$, as needed.

The existence of the non-empty Zariski open subset $U$ with the stated properties follows from the remarks above about the proof of Theorem 1.6 of \cite{KrbRes1994}.
\end{proof}

We need the following lemma whose proof is probably known to experts, but we have not been able to find a suitable reference.  

\begin{lemma}\label{lem:dimsm} If $M \subseteq F$ is a submodule of finite non-zero colength in a free module $F$ with $rank(F) = r$, then $a(S(M)) = r+d-1$. 
\end{lemma}

\begin{proof}
First, by Lemma 16.2.2(4) of \cite{SwnHnk2006} and the positivity of $d$, the Krull dimension of $S(M)$ is $r+d$. Further by (2) of the same lemma, minimal primes of $S(M)$  contract to minimal primes of $R$ and so ${\mathfrak m}S(M)$ is not contained in a minimal prime of $S(M)$ (again since $d>0$). It follows that $a(S(M)) \leq r+d-1$.

 For the reverse inequality, let $a = a(S(M))$. In \cite{Res1987} it is shown that $M$ has a minimal reduction generated by $a$ elements. (Note: in \cite{Res1987} Rees refers to $a$ as the analytic spread of $M$.) Let $N$ be a minimal reduction of $M$. We claim that $N$ has finite co-length in $F$. Suppose the claim holds. Let $\tilde{N}$ be the matrix whose columns are the generators of $N$, so that $\tilde{N}$ is an $r\times a$ matrix. By the claim, $I_r(\tilde{N})$ is $\m$-primary. By Eagon's Theorem, $d = \textrm{height}(I_r(\tilde{N})) \leq a-r+1$ (see \cite{Mts1989}, Theorem 13.10). Thus, $r+d-1\leq a$, as required. 
 
 For the claim, let $P\subsetneq \m$ be a prime ideal. Then $N_P$ is a reduction of $M_P = F_P$. By Lemma 2.4.1 in \cite{KtzRce2008}, a free module has no proper reductions. Thus, $N_P = F_P$, showing that $N$ has finite colength in $F$.
\end{proof}

\begin{proposition}\label{prop:ansp}
With notation as in Proposition \ref{prop:exist}, the analytic spread
$$
a(S(M_1\oplus \cdots \oplus M_q)) = r_1+\cdots+r_q+d-q.
$$
\end{proposition}
     
\begin{proof} Suppose that $M_k$ is generated by $m_k$ elements. Set $M = M_1 \oplus \cdots \oplus M_q \subseteq F_1 \oplus \cdots \oplus F_q = F$, and consider the ${\mathbb N}^q$-graded ring $\frac{S(M)}{{\mathfrak m}S(M)}$ whose graded component of degree $(n_1,\cdots,n_q)$  is given by $\frac{S_{n_1}(M_1)S_{n_2}(M_2) \cdots S_{n_q}(M_q)}{{\mathfrak m} S_{n_1}(M_1)S_{n_2}(M_2) \cdots S_{n_q}(M_q)}$. 

By restricting to the diagonal, the analytic spread of $S(M)$ is one more than the  dimension of the ${\mathbb N}$-graded ring 
$$
\frac{\Delta S(M)}{{\mathfrak m} \Delta S(M)} = \bigoplus_{n=0}^\infty \frac{S_n(M_1)S_n(M_2) \cdots S_n(M_d)}{{\mathfrak m} S_n(M_1)S_n(M_2) \cdots S_n(M_d)},
$$
and therefore also one more than the degree of the polynomial that gives
$$\mu(S_n(M_1)S_n(M_2) \cdots S_n(M_d))$$ for large $n$. We show that this degree equals $r_1+\cdots+r_q+d-q-1$. 
Let
$$f(n_1,\cdots,n_q) = dim_k \left( \frac{S_{n_1}(M_1)S_{n_2}(M_2) \cdots S_{n_q}(M_q)}{{\mathfrak m} S_{n_1}(M_1)S_{n_2}(M_2) \cdots S_{n_q}(M_q)}\right).$$
Standard results about ${\mathbb N}^q$-graded rings  - see Theorem 2.1.7 of \cite{Rbr1998} - imply that $f(n_1,\cdots,n_q)$ is given by a polynomial function in $n_1,\cdots,n_q$ for all $n_1,\cdots,n_q$ sufficiently large which is of degree  $t_1 \leq m_1-1$ in $n_1$, $\cdots$,  $t_q \leq m_q-1$ in $n_q$.

Next, regard  $\frac{S(M)}{{\mathfrak m}S(M)}$ as an ${\mathbb N}$-graded module with $n$th-degree component  given by
$$
\bigoplus_{n_1+\cdots+n_q= n} \frac{S_{n_1}(M_1)S_{n_2}(M_2) \cdots S_{n_q}(M_q)}{{\mathfrak m} S_{n_1}(M_1)S_{n_2}(M_2) \cdots S_{n_q}(M_q)},
$$
whose dimension is given by
$$
\sum_{n_1+\cdots+n_q= n} f(n_1,\cdots,n_q),
$$
which is therefore  a polynomial in $n$ of degree $t_1+\cdots+t_q+q-1$, for all $n$ sufficiently large. However, by Lemma \ref{lem:dimsm}, this degree equals $ r_1+\cdots+r_q+d-2$. Thus,
$t_1+\cdots+t_q+q-1 = r_1+\cdots+r_q+d-2$.

Finally, $\mu(S_n(M_1)S_n(M_2) \cdots S_n(M_q))$ is given by $f(n,n,\cdots,n)$ which, for large $n$,  is a polynomial in $n$ of degree $t_1+\cdots+t_q = r_1+\cdots+r_q+d-q-1$, as needed.
\end{proof}
\color{black}

We show that joint reductions can be detected going modulo every minimal prime.

 \begin{proposition}\label{prop:reductiontodomain}
 Let $(R,{\mathfrak m},k)$ be a Noetherian local ring of  positive  dimension $d$ with infinite residue field. 
Let $M_1 \subseteq F_1,M_2 \subseteq F_2, \cdots,$ $ M_q \subseteq F_q$ be $R$-submodules  of  finite colength in free   $R$-modules $F_1,\cdots,F_q$ of ranks $r_1,\cdots,r_q$.
Let $B_1 \subseteq M_1,\cdots,B_q \subseteq M_q$ be submodules such that $B_k$ is $r_k$-generated. Then $(B_1,\cdots,B_q)$ is a joint reduction of $(M_1,\cdots,M_q)$
iff for every minimal prime $P$ of $R$, $(\overline{B_1},\cdots,\overline{B_q})$ is a joint reduction of $(\overline{M_1},\cdots,\overline{M_q})$ where $$\overline{B_k} = \frac{B_k+PF_k}{PF_k} \subseteq  \frac{M_k+PF_k}{PF_k} = \overline{M_k} \subseteq \frac{F_k}{PF_k} = \overline{F_k}$$ denote the modules obtained by reducing modulo $P$.\end{proposition}

\begin{proof} The joint reduction equation for $(B_1,\cdots,B_q)$ is an equality of two submodules of the free $R$-module $S_{n+1}(F_1)\cdots S_{n+1}(F_q)$. Choose bases for $F_1,\cdots,F_q$. Elements of these submodules may then be regarded as column vectors with respect to the natural monomial bases of $S_{n+1}(F_1)\cdots S_{n+1}(F_q)$. Reducing all entries of these vectors modulo (any prime ideal) $P$ gives submodules over $\frac{R}{P}$ of $S_{n+1}(\overline{F_1})S_{n+1}(\overline{F_2})\cdots S_{n+1}(\overline{F_q})$ whose equality is the joint reduction equation for $(\overline{B_1},\cdots,\overline{B_q})$, thereby proving (more than) the easier direction of the equivalence of the two conditions.

For the other direction, suppose that for every minimal prime $P$, $(\overline{B_1},\cdots,\overline{B_q})$ is a joint reduction of $(\overline{M_1},\cdots,\overline{M_q})$.
Choose $n \in {\mathbb N}$ large enough so that we have joint reduction equations modulo each $P$. Lifting these equations shows that
\begin{align*}S_{n+1}(M_1)S_{n+1}(M_2)\cdots S_{n+1}(M_q) &\subseteq B_1S_{n}(M_1)S_{n+1}(M_2)\cdots S_{n+1}(M_q)\\
 &+ S_{n+1}(M_1)B_2S_{n}(M_2)\cdots S_{n+1}(M_q)\\
 &+\cdots\\
 &+S_{n+1}(M_1)S_{n+1}(M_2)\cdots B_qS_{n}(M_q)\\
  &+ PS_{n+1}(F_1)S_{n+1}(F_2)\cdots S_{n+1}(F_q)
\end{align*}
for each minimal prime $P$.
If the number of minimal primes is $s$ and ${\mathfrak n}^t = 0$ where ${\mathfrak n}$ is the nilradical of $R$, then we will show that $(B_1,\cdots,B_q)$ is a joint reduction of $(M_1,\cdots,M_q)$ with joint reduction number at most $N = (n+1)st-1$.

Consider a generator of $S_{N+1}(M_1)S_{N+1}(M_2)\cdots S_{N+1}(M_q)$. Regard it as a product of $N+1$ linear forms in the basis of $F_1$, $N+1$ linear forms in the basis of $F_2$, $\cdots$, $N+1$ linear forms in the basis of $F_q$ - the forms corresponding to generators of $M_1,\cdots,M_q$. Explicitly write it as $L = L^1_1\cdots L^1_{N+1}L^2_1\cdots L^2_{N+1}\cdots L^q_1\cdots L^q_{N+1}$ and regroup as $H_1\cdots H_t$ where each $H_i$ is a product of $(n+1)s$ of the $L^1_1\cdots L^1_{N+1}$, $(n+1)s$ of the $L^2_1\cdots L^2_{N+1}$, $\cdots$, $(n+1)s$ of the $L^q_1\cdots L^q_{N+1}$.

Any single $H_i$ is a product of $s$ generators of $S_{n+1}(M_1)S_{n+1}(M_2)\cdots S_{n+1}(M_q)$, say
$Q_1\cdots Q_s$. So there exist $Q_1^\prime, \cdots, Q_s^\prime$ in
\begin{align*}& B_1S_{n}(M_1)S_{n+1}(M_2)\cdots S_{n+1}(M_q)\\
 &+ S_{n+1}(M_1)B_2S_{n}(M_2)\cdots S_{n+1}(M_q)\\
 &+\cdots\\
 &+S_{n+1}(M_1)S_{n+1}(M_2)\cdots B_qS_{n}(M_q)
\end{align*}
such that $Q_j - Q_j^\prime \in P_jS_{n+1}(F_1)S_{n+1}(F_2)\cdots S_{n+1}(F_q)$. Multiply all these equations over $j=1,\cdots,s$ in $S(F)$ and collect terms to see that there exists $H_i^\prime$ in
\begin{align*}& B_1S_{(n+1)s-1}(M_1)S_{(n+1)s}(M_2)\cdots S_{(n+1)s}(M_q)\\
 &+ S_{(n+1)s}(M_1)B_2S_{(n+1)s-1}(M_2)\cdots S_{(n+1)s}(M_q)\\
 &+\cdots\\
 &+S_{(n+1)s}(M_1)S_{(n+1)s}(M_2)\cdots B_qS_{(n+1)s-1}(M_q)
\end{align*}
such that $H_i-H_i^\prime \in {\mathfrak n}S_{n+1}(F_1)S_{n+1}(F_2)\cdots S_{n+1}(F_q)$. 

Finally, multiplying all these equations over $i=1,\cdots,t$ and collecting terms shows that $L = H_1\cdots H_t$ is in
\begin{align*} &B_1S_{N}(M_1)S_{N+1}(M_2)\cdots S_{N+1}(M_q)\\
 &+ S_{N+1}(M_1)B_2S_{N}(M_2)\cdots S_{N+1}(M_q)\\
 &+\cdots\\
 &+S_{N+1}(M_1)S_{N+1}(M_2)\cdots B_qS_{N}(M_q),
\end{align*}
as desired.
\end{proof}

We have chosen an equational definition of joint reductions. We  now show that it is equivalent to the valuative and determinantal definitions. The notation $det(B_k)$ in the next proposition refers to the determinant of a matrix whose columns represent a system of generators of $B_k$ in terms of a basis of $F_k$. It is determined up to unit multiple by the submodule $B_k \subseteq F_k$.

\begin{theorem}\label{thm:equiv}
Let $(R,{\mathfrak m},k)$ be a Noetherian local ring of  positive  dimension $d$ with infinite residue field. 
Let $M_1 \subseteq F_1,M_2 \subseteq F_2, \cdots,$ $ M_q \subseteq F_q$ be $R$-submodules  of  finite colength in free   $R$-modules $F_1,\cdots,F_q$ of ranks $r_1,\cdots,r_q$.
Let $B_1 \subseteq M_1,\cdots,B_q \subseteq M_q$ be submodules such that $B_k$ is $r_k$-generated for $k=1,2,\cdots,q$. The following are equivalent.
\begin{enumerate}
\item Equational: $(B_1,\cdots,B_q)$ is a joint reduction of $(M_1,\cdots,M_q)$ as in Definition \ref{def:jtred}.
\item Valuative: For every minimal prime $P$ of $R$ and discrete valuation ring $V$ between $R/P$ and its field of fractions, at least one $\overline{B_k}V = \overline{M_k}V$, with $\overline{B_k}$, $\overline{M_k}$ as in Proposition \ref{prop:reductiontodomain}.
\item Determinantal: $(det(B_1),\cdots,det(B_q))$ is a joint reduction of $(I_1,\cdots,I_q)$ where $I_i$ $( = I(M_i))$ denotes the $0$th Fitting ideal of $F_i/M_i$ (which is the ideal of $r_i$-sized minors of a matrix whose columns represent generators of $M_i$ with respect to a basis of $F_i$).
\end{enumerate}
\end{theorem}

\begin{proof} Using Proposition \ref{prop:reductiontodomain} we may assume that $R$ is a domain (of positive dimension) with field of fractions $K$ in which case (2) just says that for any discrete valuation ring $V$ between $R$ and $K$, at least one $B_kV = M_kV$.\\
(1) $\Rightarrow$ (2): Let ${\mathfrak n} = tV$ be the maximal ideal of $V$. Extend the joint reduction equation to $V$ and abuse notation to denote the extensions of various modules to $V$ by the same symbols. For instance, $M_1$ will now denote $M_1V$ which is a free $V$-module of rank $r_1$.
Choose appropriate bases $\{T_{11},\cdots,T_{1r_1}\}$ for $M_1$, $\{T_{21},\cdots,T_{2r_2}\}$ for $M_2$ and so on, so that $B_1$, $B_2$, $\cdots$, $B_q$ are diagonal with respect to these bases. Say $B_1$ has basis $\{t^{m_{11}}T_{11},\cdots,t^{m_{1r_1}}T_{1r_1}\}$, $B_2$ has basis $\{t^{m_{21}}T_{21},\cdots,t^{m_{2r_2}}T_{2r_2}\}$ and so on. Here the $m_{kj} \in {\mathbb N}\cup \{\infty\}$ with $m_{kj} = \infty$ meaning that $t^{m_{kj}} = 0$.

A basis of $S_{n+1}(M_1)S_{n+1}(M_2)\cdots S_{n+1}(M_q)$ is clearly given by all monomials of the form $$T_{11}^{a_{11}}\cdots T_{1r_1}^{a_{1r_1}}T_{21}^{a_{21}}\cdots T_{2r_2}^{a_{2r_2}}\cdots T_{q1}^{a_{q1}}\cdots T_{qr_q}^{a_{qr_q}}$$ where $a_{11}+\cdots+a_{1r_1} = a_{21}+\cdots+a_{2r_2} = \cdots = a_{q1}+\cdots+a_{qr_q} = n+1$. 

A little thought shows that a basis of $B_1S_{n}(M_1)S_{n+1}(M_2)\cdots S_{n+1}(M_q)$ is given by  all terms of the form 
$$T_{11}^{a_{11}}\cdots T_{1r_1}^{a_{1r_1}}T_{21}^{a_{21}}\cdots T_{2r_2}^{a_{2r_2}}\cdots T_{q1}^{a_{q1}}\cdots T_{qr_q}^{a_{qr_q}}t^{ min \{m_{1j}: a_{1j}>0\} }$$ where, as before,  $a_{11}+\cdots+a_{1r_1} = a_{21}+\cdots+a_{2r_2} = \cdots = a_{q1}+\cdots+a_{qr_q} = n+1$. Similarly a basis of $S_{n+1}(M_1)S_{n+1}(M_2)\cdots B_kS_{n}(M_k)\cdots S_{n+1}(M_q)$ is given by  all terms of the form 
$$T_{11}^{a_{11}}\cdots T_{1r_1}^{a_{1r_1}}T_{21}^{a_{21}}\cdots T_{2r_2}^{a_{2r_2}}\cdots T_{q1}^{a_{q1}}\cdots T_{qr_q}^{a_{qr_q}}t^{ min \{m_{kj}: a_{kj}>0\} }$$ where again, $a_{11}+\cdots+a_{1r_1} = a_{21}+\cdots+a_{2r_2} = \cdots = a_{q1}+\cdots+a_{qr_q} = n+1$.

Thus if some $m_{1j_1} >0$, some $m_{2j_2}>0$, $\cdots$, some $m_{qj_q} >0$, then, $T_{1j_!}^nT_{2j_2}^n\cdots T_{qj_q}^n$ is in  $S_{n}(M_1)\cdots S_{n}(M_q)$ but not to the right hand side of the joint reduction equation.
So for the joint reduction equation to hold when extended to $V$, all $m_{kj}$ must vanish for some $k$, and then for this $k$, $B_k=M_k$ or, reverting to the original notation, the extended module $B_kV = M_kV$.  \\
(2) $\Rightarrow$ (3): Let $V$ be a discrete valuation ring between $R$ and $K$. To say that $B_kV = M_kV \subseteq F_kV$ is equivalent to saying that $det(B_k)V = I_kV$  and therefore by (2), this holds for some $k$. Hence
$$
I_1I_2 \cdots I_qV = (det(B_1)I_2I_3\cdots I_q + det(B_2)I_1I_3I_4 \cdots I_q + \cdots + det(B_q)I_1I_2\cdots I_{q-1})V.
$$
Since this is true for every discrete valuation ring between $R$ and $K$, it follows that $(det(B_1),\cdots,det(B_q))$ is a joint reduction of $(I_1,\cdots,I_q)$.\\
(3) $\Rightarrow$ (2):  If  $(det(B_1),\cdots,det(B_q))$ is a joint reduction of $(I_1,\cdots,I_q)$ then the equation above holds for each discrete valuation ring $V$ between $R$ and its field of fractions. If each $det(B_k)V \subsetneq I_kV$, then $det(B_k)V \subseteq {\mathfrak n}I_kV$ and so this equation could not possibly hold. Thus some $det(B_k)V = I_kV$ or equivalently $B_kV = M_kV$.\\
(2) $\Rightarrow$ (1): Consider the submodules 
$$
N = B_1M_2\cdots M_q + M_1B_2M_3 \cdots M_q + \cdots + M_1\cdots M_{q-1}B_q \subseteq M_1M_2\cdots M_q
$$ 
of $F_1\cdots F_q$. From (2) it follows that these are are equal when extended to any discrete valuation ring $V$ between $R$ and $K$, or that $N$ is a reduction of $M_1\cdots M_q$  in the sense of \cite{Res1987}. By Theorem 1.5(ii) of \cite{Res1987} it follows that the ideal generated by $N$ in $S(M_1\cdots M_q) \subseteq S(F_1\cdots F_q)$ is irrelevant. So there is an $n \in {\mathbb N}_0$ such that $S_{n+1}(M_1\cdots M_q) = NS_n(M_1\cdots M_q)$.

The ${\mathbb N}$-graded $R$-algebra $\Delta S(F)$ is generated over $R$ by $(\Delta S(F))_1 = F_1\cdots F_q$ and so there is a natural surjective $R$-algebra  homomorphism
from $S(F_1\cdots F_q)$ onto $\Delta S(F)$. Taking the images under this homomorphism of the equality in the last line of the previous paragraph gives the required joint reduction equation for $(B_1,\cdots,B_q)$.
\end{proof}

The next property of joint reductions we would like to show is that minimal generating sets of $B_k$ can be extended to minimal generating sets of $M_k$, and that when $R$ has positive depth, each $B_k$ can be chosen to be a free module (see Theorem \ref{thm:properties} below). For this, we need a few preliminary results. 

\begin{lemma}
Let $k$ be an infinite field. For $r \leq n$, set $N=\binom{n}{r}$. Consider the map $$M_{r \times n}(k) \rightarrow k^N: A \mapsto (A_I)_{I \subseteq [n],|I|=r}$$ given by taking an $r \times n$ matrix $A$ over $k$ to the $N$-vector of all its maximal minors. The image of this map is not contained in any union of finitely many proper linear subspaces of $k^N$.
\end{lemma}

\begin{proof} The image of the map above is the cone over the Grassmannian $G_{r,n}$ and is the affine variety, say $Z$, in $k^N$ defined by the Plucker relations - see Section 2 of \cite{KlmLks1972}. Further $Z$ is not contained in any hyperplane. For, any hyperplane is defined by an equation $\sum_I \alpha_I X_I =0$ where not all the $\alpha_I$ vanish. If $\alpha_I \neq 0$ where $I = \{n_1 < \cdots < n_r\}$ then the image of the matrix $A$ whose $n_1,\cdots,n_r$ columns form the identity matrix (of size $r$) and whose other columns vanish is not contained in $Z$. So the intersection of any hyperplane with $Z$ is a proper Zariski closed subset of $Z$. If $Z$ were to be contained in a finite union of hyperplanes, then, taking inverse images,
$M_{r \times n}(k)$ would be a finite union of proper closed subsets. But $M_{r \times n}(k)$ is irreducible over an infinite field.
\end{proof}

\begin{lemma}
Let $J$ be an ideal generated by $z_1,\cdots,z_m$ and $P$ be a prime ideal that does not contain $J$. There is a  proper linear subspace $L$ in $k^m$ such that if $(\overline{a}_1,\cdots,\overline{a}_m) \notin L$, then $a_1z_1+\cdots+a_mz_m$ is not in $P$.
\end{lemma}

\begin{proof} If $a_1z_1+\cdots+a_mz_m \in P$, then, $a_1z_1+\cdots+a_mz_m \in P \cap J$ which is an ideal properly contained in $J$. So
$$\overline{a}_1\overline{z}_1+\cdots+\overline{a}_m\overline{z}_m \in \frac{(P \cap J) + {\mathfrak m}J}{{\mathfrak m}J}$$
which is a proper subspace of $\frac{J}{{\mathfrak m}J}$. The collection of $(\overline{a}_1,\cdots,\overline{a}_m)$ that satisfy the above condition clearly form a linear subspace $L$ of $k^m$ which is proper since $\overline{z}_1,\cdots,\overline{z}_m$ span $\frac{J}{{\mathfrak m}J}$. Now if $(\overline{a}_1,\cdots,\overline{a}_m) \notin L$, then $a_1z_1+\cdots+a_mz_m$ is not in $P$, as needed.
\end{proof}

\begin{proposition}
Let $(R,{\mathfrak m},k)$ be a Noetherian local ring of positive depth and $M \subseteq F$ be a submodule of a finitely-generated free module $F$ of rank $r$ such that $J = I(M)$ is ${\mathfrak m}$-primary. Suppose that $M$ is generated by $\{m_1,\cdots,m_n\}$. Then there is a non-empty Zariski open subset $U$ of $M_{n \times r}(k)$ such that if $\overline{A} \in U$ (for $A \in M_{n \times r}(R)$), then the determinant of $[m_1 ~ \cdots ~ m_n]A$ (which is an $r \times r$ matrix over $R$) is a non-zero-divisor in $R$.
\end{proposition}

\begin{proof} For an $r$-subset $I$ of $[n]$ let $z_I$ be the determinant of the matrix formed by the $I$ columns of $\{m_1,\cdots,m_n\}$ and let ${A}_I$ be the determinant of the matrix formed by the $I$ rows of ${A}$. The ideal $J$ is generated by all the $z_I$. Also it is not contained in any associated prime of $(0)$, since $depth(R) >0$.

For each associated prime $P$ of $(0)$ there is a proper linear subspace of $k^N$ (where $N =\binom{n}{r}$) such that if $(\overline{a}_I)_{I \subseteq [n],|I|=r}$ is not in that subspace, then, $\sum_I a_Iz_I \notin P$. Let $V$ be the complement of the union of these subspaces for all associated primes of $(0)$, so that $V$ is a non-empty Zariski open subset of $k^N$ such that if
$(\overline{a}_I)_{I \subseteq [n],|I|=r} \in V$, then, $\sum_I a_Iz_I$ is a non-zero-divisor.

Now consider the map $$M_{r \times n}(k) \rightarrow k^N: A \mapsto (A_I)_{I \subseteq [n],|I|=r}.$$
The image of this map necessarily intersects $V$, so that the preimage $U$ of $V$ is a non-empty Zariski open subset of $M_{r \times n}(k)$.  But by Cauchy-Binet,
\[
\pushQED{\qed} 
det([m_1 ~ \cdots ~ m_n]A) = \sum_I z_IA_I. \qedhere
\]     
\end{proof}

\begin{theorem}\label{thm:properties} With the notation of Theorem \ref{thm:equiv}, suppose that $q=d$ and that all $M_k \subsetneq F_k$. Then, each $det(B_k)$ is part of a minimal generating set of $I_k$ and minimal generating sets of $B_k$ can be extended to those of $M_k$. If $q\geq d$ and $R$ has positive depth, then each $B_k$ can be chosen to be a free submodule of $F_k$.
\end{theorem}

\begin{proof}
 If each $M_k$ is a proper submodule of $F_k$, then each ideal $I_k$ is ${\mathfrak m}$-primary. Whenever $x_1,\cdots,x_d$ is a joint reduction of ${\mathfrak m}$-primary ideals $(I_1,\cdots,I_d)$, each $x_k$ is necessarily a minimal generator of $I_k$, else the ${\mathfrak m}$-primary ideal $I_1\cdots I_d$ would be contained in an ideal generated by less that $d$ elements by Nakayama's lemma. Hence each $det(B_k)$ is part of a minimal generating set of $I_k$.
It follows that $det(B_k) \neq 0$. If an $R$-linear combination of elements of $B_k$ is in ${\mathfrak m}M_k$, and at least one of the coefficients is a unit, then that element may be replaced by the element of ${\mathfrak m}M_k$ (without changing $B_k$) at the expense of changing $det(B_k)$ by a multiplicative unit. However now, after this change, $det(B_k)$ is in ${\mathfrak m}I_k$, and therefore, not part of a minimal generating set of $I_k$, which is a contradiction. Thus a minimal generating set of $B_k$ can be extended to one of $M_i$.

\medskip
Now suppose $R$ has positive depth. We may assume that $M_k \subsetneq F_k$. We will need to appeal to the proof of Theorem 1.6  of Section~1 of \cite{KrbRes1994}. Applied to the ${\mathbb N}^q$-graded algebra $S(M)$ where $M = M_1 \oplus \cdots \oplus M_q$, this theorem proves the following: Suppose that $M_k$ is $n_k$-generated and identify $M_k$ with the $r_k \times n_k$ matrix whose columns are these generators. Choose generic matrices $A_1,\cdots,A_q$  of sizes $n_1 \times r_1,\cdots,n_q \times r_q$ - each entry of these is a distinct indeterminate. Consider the product matrices $M_kA_k$ which are square of sizes $r_k \times r_k$. There exists a non-empty Zariski open subset, say $W$,  of $M_{n_1 \times r_1}(k) \times \cdots \times M_{n_1 \times r_1}(k)$ such that 
if $(\overline{A}_1,\cdots,\overline{A}_q) \in W$, then, $B_k = M_kA_k$ are such that $(B_1,\cdots,B_q)$  is a joint reduction of $(M_1,\cdots,M_q)$. 

On the other hand, for each $k$, there is a non-empty Zariski open subset $U_k$ of $M_{n_k \times r_k}(k)$ such that if $\overline{A}_k \in U_k$ then the determinant of $B_k = M_kA_k$ is a non-zero-divisor in $R$. Thus $B_k$ is a free submodule of $M_k$.

So for any element $(\overline{A}_1,\cdots,\overline{A}_q) \in W \cap (U_1 \times \cdots \times U_q)$, the collection $(B_1,\cdots,B_q)$  is a joint reduction of $(M_1,\cdots,M_q)$ and each $B_k$ is free, as desired.
\end{proof}

\section{Reduction-number-zero for two-dimensional regular local rings}

We  use the ideas of  \S1 to prove a reduction-number-zero theorem for integrally closed modules over two-dimensional regular local rings. Throughout this section  $(R,{\mathfrak m},k)$ is a two-dimensional regular local ring with infinite residue field. 
For an ideal $I$ of $R$, the largest power of ${\mathfrak m}$ containing $I$ is denoted by $ord(I)$ and for a torsion-free module $M$, set $ord(M) = ord(I(M))$ where $I(M)$ is the ideal of maximal minors of (a matrix whose columns generate) $M \subseteq F$ where $F=M^{**}$. We will use $\nu(\cdot)$ to denote the minimal number of generators function. A  necessary and sufficient condition for a torsion-free module $M$ to be contracted from an overring of $R$ of the form $R[\frac{\mathfrak m}{x}]$ for a minimal generator $x \in {\mathfrak m}$ is that
$\nu(M) = ord(M)+rank(M)$, while in general, $\nu(M) \leq ord(M)+rank(M)$. In particular, the equality holds for integrally closed modules. For such basic properties of integrally closed modules over two-dimensional regular local rings, we refer to \cite{Kdy1995}.

\begin{theorem}\label{thm:jtred0} 
Let $M_1 \subseteq F_1$ and $M_2 \subseteq F_2$  be integrally closed $R$-modules of finite colength in free $R$-modules of rank $r_1$ and $r_2$ respectively
and let $(B_1,B_2)$ be a joint reduction of $(M_1,M_2)$. Then, as submodules of $S(M)$,
$$
M_1M_2 = B_1M_2+M_1B_2.
$$
Equivalently, the joint reduction number of $(M_1,M_2)$ with respect to $(B_1,B_2)$ is 0. 
\end{theorem}

\begin{proof} We first dispose of the case that either $M_1=F_1$ or $M_2=F_2$, in which case either $I_1 = I(M_1) = R$ or $I_2 = I(M_2) = R$. Without loss of generality, assume that $M_1=F_1$ so that $I_1 =R$. Then, Theorem \ref{thm:equiv} shows that $(det(B_1),det(B_2))$ is a joint reduction of $(R,I_2)$ or equivalently, $det(B_2)R+det(B_1)I_2$ is a reduction of $I_1I_2 = I_2$. If $det(B_1)$ is a non-unit, then $det(B_2)R$ is a reduction of $I_2$. Thus $I_2=R$ and $det(B_2)$ is a unit.  So $det(B_1)$ or $det(B_2)$ is necessarily a unit which implies that either $B_1 = (M_1 =)\ F_1$ or $B_2 = (M_2 =)\ F_2$. Either of these conditions implies that $M_1M_2 = B_1M_2+M_1B_2$, as needed.

We will now treat the case that $M_1$ and $M_2$ are proper submodules of $F_1$ and $F_2$ respectively. Suppose that $ord(M_1) = n_1$ and $ord(M_2) = n_2$. 
Then Theorem~\ref{thm:properties} shows that $B_1$ and $B_2$ are free submodules of $M_1$ and $M_2$ of ranks $r_1$ and $r_2$ respectively (since $det(B_1)$ and $det (B_2)$ are non-zero), minimal generating sets of which extend to those of $M_1$ and $M_2$. Further, for 
 some $n \in {\mathbb N}$, the joint reduction equation 
\begin{eqnarray*}
S_n(M_1)S_n(M_2) &=& B_1S_{n-1}(M_1)S_n(M_2) + S_n(M_1)B_2S_{n-1}(M_2),
\end{eqnarray*}
holds in $S(M)$ where $M=M_1 \oplus M_2$. Then, taking $V$ to be the order valuation ring of $R$ in Theorem \ref{thm:equiv}(2), either $ord(det(B_1)) = n_1$ or $ord(det(B_2)) = n_2$.

We  now prove the following statements. 
\begin{enumerate}
\item $B_1M_2+M_1B_2$ is of finite colength in $F_1F_2$ and in particular is of rank  $r_1r_2$. \item $\nu(B_1M_2+M_1B_2) = r_1r_2+r_1n_2+n_1r_2$.
\item $ord(B_1M_2+M_1B_2) \leq r_1n_2+n_1r_2$.
\end{enumerate}

{\it Proof of (1).} It suffices to show that if $P$ is a height one prime ideal of $R$ and $V$ is the discrete valuation ring $R_P$, then
$$
(B_1M_2+M_1B_2)_P = (F_1F_2)_P.
$$
Localising the joint reduction equation at $P$ and,  again, abusing notation by using the same symbols to denote the modules extended to $V$,  note that $M_1=F_1$  and $M_2=F_2$ by the finite length assumption. The joint reduction equation then becomes
\begin{eqnarray*}
S_n(F_1)S_n(F_2) &=& B_1S_{n-1}(F_1)S_n(F_2) + B_2S_n(F_1)S_{n-1}(F_2),
\end{eqnarray*}
which says that  $(B_1,B_2)$  is a joint reduction of $(F_1,F_2)$ over $V$.  By Theorem~\ref{thm:equiv}(2), either $B_1 = F_1$ or $B_2=F_2$. This implies that
$$
B_1F_2+F_1B_2= F_1F_2,
$$
as was to be seen.

{\it Proof of (2).}
Extend $B_1$ and $B_2$ to minimal generating sets of $M_1$ and $M_2$ respectively and regard all these generators as linear forms in some chosen bases $\{T_{11},\cdots,T_{1r_1}\}$ for $F_1$ and $\{T_{21},\cdots,T_{2r_2}\}$ for $F_2$. Suppose we get linear forms $G_1,\cdots,G_{r_1+n_1}$ for $M_1$ and 
$H_1,\cdots,H_{r_2+n_2}$ for $M_2$ where the first $r_1$ of the $G$'s are a basis of $B_1$ and the first $r_2$ of the $H$'s are a basis of  $B_2$.

Then, $B_1M_2+M_1B_2$ is generated by the products $G_jH_l$ where either $1 \leq j \leq r_1$ and $1 \leq l \leq r_2+n_2$ or $r_1+1 \leq j \leq r_1+n_1$ and $1 \leq l \leq r_2$ - a total of $r_1r_2+r_1n_2+n_1r_2$ generators. We show that these form a minimal set of generators.
Suppose that for some $\alpha_{jl},\beta_{jl} \in R$, we have
$$
\sum_{1 \leq j \leq r_1, 1 \leq l \leq r_2+n_2} \alpha_{jl} G_jH_l = \sum_{1 \leq j \leq n_1, 1 \leq l \leq r_2} \beta_{jl}G_{r_1+j}H_l.
$$
We need to see that all of $\alpha_{jl},\beta_{jl}$ are in ${\mathfrak m}$. This is equivalent to $r_1r_2$ equations - one each for the coefficient of $T_{1i}T_{2k}$ for $1 \leq i \leq r_1, 1 \leq k \leq r_2$.
Writing $G_j = \sum_{i=1}^{r_1} T_{1i}B_1(i,j)$ and $H_l = \sum_{k=1}^{r_2} T_{2k}B_2(k,l)$, 
the previous equation reads:
\begin{eqnarray*}
\sum_{1 \leq i,j \leq r_1, 1 \leq l \leq r_2+n_2,1 \leq k \leq r_2} \alpha_{jl} T_{1i}B_1(i,j)T_{2k}M_2(k,l) = \\
\sum_{1 \leq i \leq r_1,1 \leq j \leq n_1, 1 \leq k,l \leq r_2} \beta_{jl}T_{1i}M_1(i,r_1+j)T_{2k}B_2(k,l).
\end{eqnarray*}
Equating the coefficients of $T_{1i}T_{2k}$ on both sides gives:
\begin{eqnarray*}
\sum_{1 \leq j \leq r_1, 1 \leq l \leq r_2+n_2} \alpha_{jl} B_1(i,j)M_2(k,l) = 
\sum_{1 \leq j \leq n_1, 1 \leq l \leq r_2} \beta_{jl}M_1(i,r_1+j)B_2(k,l).
\end{eqnarray*}

In matrix form, this is the equality
$$
B_1 \alpha M_2^T = M_1^\prime \beta B_2^T
$$
of $r_1 \times r_2$ matrices where $\alpha$ is a matrix of size $r_1 \times (r_2+n_2)$ and $\beta$ is a matrix of size $n_1 \times r_2$. 
The notation $M_1^\prime$ is for the submatrix of $M_1$ of the last $n_1$ columns.

By left multiplication by $adj(B_1)$ and right multiplication by $adj(B_2)^T$, we get the equation
$$
det(B_1) \alpha M_2^T adj(B_2)^T = adj(B_1) M_1^\prime \beta det(B_2).
$$
Since $(det(B_1),det(B_2))$  is a joint reduction of $(I_1,I_2)$ and both $I_1,I_2$ are ${\mathfrak m}$-primary ideals, they form a regular sequence in $R$. Hence there exists a $r_1 \times r_2$ matrix $C$ such that
\begin{eqnarray*}
\alpha M_2^T adj(B_2)^T &=& C det(B_2), {\text {~and}}\\
adj(B_1) M_1^\prime \beta&=& det(B_1) C.
\end{eqnarray*}
and therefore also such that
\begin{eqnarray*}
\alpha M_2^T  &=& C B_2^T \Rightarrow M_2 \alpha = B_2 C^T, {\text {~and}}\\
 M_1^\prime \beta &=& B_1 C.
\end{eqnarray*}

Since the columns of $B_1$ and the columns of $M_1^\prime$ together form a minimal generating set for $M_1$, the second of the equations above implies that all entries of $\beta$ and $C$ are in ${\mathfrak m}$. Then using that the columns of $M_2$ form a minimal generating set for $M_2$ (and that the entries of $C$ are in ${\mathfrak m}$), the first of the equations above implies that all entries of $\alpha$ are also in ${\mathfrak m}$, just as needed.

{\it Proof of (3).}
Since $B_1M_2+M_1B_2$ is of finite colength in the free module $F_1F_2$ of rank $r_1r_2$ by (1), $ord(B_1M_2+M_1B_2)$ is the order of the ideal $I_{r_1r_2}(B_1M_2+M_1B_2)$. This ideal contains the ideals $I_{r_1r_2}(B_1M_2)$ and $I_{r_1r_2}(M_1B_2)$. So it suffices to see that at least one of these has order at most $r_1n_2+n_1r_2$. Since either $ord(det(B_1)) = n_1$ or $ord(det(B_2)) = n_2$, say $ord(det(B_1)) = n_1$ without loss of generality.

We prove that $I_{r_1r_2}(B_1M_2)$ contains an element of order $r_1n_2+n_1r_2$. Simply pick $r_2$ columns of $M_2$, such that if $N_2$ is the square matrix they form, then $det(N_2)$ has order $n_2$. Now $det(B_1 \otimes N_2) = det(B_1)^{r_2}det(N_2)^{r_1}$ is in $I_{r_1r_2}(B_1M_2)$ and has the required order.

We now complete the proof. Given (1)-(3), it follows immediately by the numerical characterisation of contracted modules that $B_1M_2+M_1B_2$ is contracted, and so to see that it equals $M_1M_2$, it suffices to see this after extending to a first quadratic transform of $R$. Since the extension of an integrally closed module to a first quadratic transform of $R$ is also integrally closed (Proposition 4.6 of \cite{Kdy1995}) and is either free or has strictly smaller Buchsbaum-Rim multiplicity (Theorem 4.8 of \cite{Kdy1995}) we may assume by induction on the Buchsbaum-Rim multiplicity that one of $M_1,M_2$ is actually free. But this is the case we treated at the outset.
\end{proof}

\begin{remark}
There is a version of Theorem \ref{thm:jtred0} that is valid for more than two integrally closed modules. For $q \geq 2$, suppose that $M_1 \subseteq F_1$, $\cdots$, $M_q \subseteq F_q$  are integrally closed $R$-modules of finite colength in free $R$-modules $F_1,\cdots,F_q$ of ranks $r_1, \cdots,r_q$  respectively.
Let $(B_1,\cdots,B_q)$ be a joint reduction of $(M_1,\cdots,M_q)$ such that for each subset $K$ of $[q]$ with cardinality at least 2, the collection $\{B_k:k \in K\}$ is a joint reduction of $\{M_k:k \in K\}$. The existence of such a joint reduction follows easily from the Zariski openness condition stated in Proposition \ref{prop:exist}.
Then, 
\begin{eqnarray*}
M_1 \cdots M_q &=& B_1 \cdots B_{q-1}M_q+ \cdots +M_1 B_2 \cdots B_q.
\end{eqnarray*}
This follows using Theorem \ref{thm:jtred0} and induction. It would be interesting to decide if the same conclusion holds assuming only that $(B_1,\cdots,B_q)$ is a joint reduction of $(M_1,\cdots,M_q)$.
\end{remark}

\section{Mixed Buchsbaum-Rim multiplicity of modules and Koszul complexes}

Throughout this section $(R,{\mathfrak m},k)$ will be a Noetherian local ring of positive dimension $d$ with infinite residue field. Let $F_k$ be a free $R$-module of rank $r_k$ and $M_k \subseteq F_k$ be modules of finite colength for $k=1,2,\cdots,d$. Our goal in this section is to define the mixed Buchsbaum-Rim multiplicity $br(M_1|\cdots|M_d)$ of this collection of modules and to express it as the Euler-Poincar\'{e} characteristic of a certain Koszul complex. To elaborate, it is known that for $n_1, \ldots, n_q$ sufficiently large, the function giving the length of $S_{n_1}(F_1)\cdots S_{n_q}(F_q)/S_{n_1}(M_1)\cdots S_{n_q}(M_q)$ assumes the values of a polynomial - the joint Buchsbaum-Rim polynomial - with rational coefficients whose total degree is $d+r_1+\cdots +r_q-q$. (See \S7 for a relevant discussion.) When $q = d$, the current case of interest, the total degree of the joint Buchsbaum-Rim polynomial $p(n_1, \ldots, n_d)$ is $r_1+\cdots +r_d$ and we focus on the normalised coefficient of the monomial $n_1^{r_1}\cdots n_d^{r_d}$ term in this polynomial.

\begin{definition}
With notation above, $r_1!\cdots r_d!$ times the coefficient of $n_1^{r_1} \cdots n_d^{r_d}$ in the joint Buchsbaum-Rim polynomial $p(n_1,\cdots,n_d)$  
is called the mixed Buchsbaum-Rim multiplicity of $M_1,\cdots,M_d$ and is denoted by $br(M_1|\cdots|M_d)$.
\end{definition}

As mentioned in the introduction, several authors have considered various versions of mixed Buchsbaum-Rim multiplicities for modules, since the coefficients of the leading form of the joint Buchsbaum-Rim polynomial of any collection of finite colength modules are of considerable interest. Our focus on $br(M_1|\cdots |M_d)$ derives from its relation to our definition of joint reduction. Note that when $M_1 = I_1, \ldots, M_d = I_d$ are $\m$-primary ideals, then $br(I_1|\cdots |I_d)$ is the mixed multiplicity $e(I_1|\cdots |_d)$ of $I_1, \ldots, I_d$ as defined by Rees in \cite{Res1984}. Rees shows that if $x_1, \ldots, x_d$ is a joint reduction of $I_1, \ldots, I_d$, then $br(I_1|\cdots |I_d)$ equals the multiplicity of the ideal $(x_1, \ldots, x_d)$, which in turn equals the Euler-Poincar\'{e} characteristic of the Koszul complex on $x_1, \ldots, x_d$. 

We now describe a complex of free $R$-modules constructed from a joint reduction of $(M_1,\cdots,M_d)$ whose Euler-Poincar\'{e} characteristic computes $br(M_1|\cdots|M_d)$. Let $(B_1,\cdots,B_d)$ be a joint reduction of $(M_1,\cdots,M_d)$. Let 
$\phi_k: F_k \rightarrow F_k$ be an endomorphism, for  $k=1,2,\cdots,d$ such that $im(\phi_k) = B_k$. 
 We will regard this as a complex of free $R$-modules denoted $K_{\bullet}(\phi_k)$:
 \begin{center}
\begin{tikzcd}
 0  \arrow[r] &  F_k   \arrow[r,"\phi_k"]  & F_k   \arrow[r]  & 0,
\end{tikzcd}
\end{center}
for $k=1,2,\cdots,d$. We set $K_{\bullet}(\phi_1, \ldots, \phi_d)= K_{\bullet}(\phi_1)\otimes \cdots \otimes K_{\bullet}(\phi_d)$. It follows from Theorem \ref{thm:comparison} below, that the homology modules in $K_{\bullet}(\phi_1, \ldots, \phi_d)$ have finite length if and only if $det(\phi _1), \ldots, det(\phi_d)$ generate an $\m$-primary ideal, and this latter condition holds since $B_1, \ldots B_d$ is a joint reduction of $M_1, \ldots, M_d$ (see Theorem \ref{thm:equiv}).

The main result of this section is the following theorem which is a version for modules of Rees's theorem mentioned above. The proof that we give below seems to be new even in the case of ideals.

\begin{theorem}\label{thm:brmultepchar} With notation as above, assume that $(B_1, \ldots, B_d)$ is a joint reduction of $(M_1, \ldots, M_d)$ and each $M_k \subsetneq F_k$. Then,
$$
br(M_1|\cdots|M_d) = \chi(K_{\bullet}(\phi_1,\cdots,\phi_d)).
$$
\end{theorem}

The strategy for the proof this theorem is to introduce auxiliary complexes $K_{\bullet} (\LL; \RR)$ and $K_{\bullet}(\LL; \s)$, defined as follows: For each $k=1,2,\cdots,d$, choose bases $T_{k1},\cdots,T_{kr_k}$ for $F_k$ and sets of generators ${\mathcal L}_k = \{L_{k1},\cdots,L_{kr_k}\}$ for $B_k$. Regard the $L_{ij}$ as linear forms in the $T_{ij}$  belonging to the ${\mathbb N}^d$-graded ring ${\mathcal R} = S(M) \subseteq S(F) =: {\mathcal S}$ 
where $M = M_1 \oplus \cdots \oplus M_d \subseteq F_1 \oplus \cdots \oplus F_d =F$. 
Let ${\mathcal L}$ denote the set ${\mathcal L}_1 \coprod \cdots \coprod {\mathcal L}_d$. Note that the elements of ${\mathcal L}_k$ are homogeneous of degree $e_k \in {\mathbb N}^d$. The proof is based on studying the inclusions of Koszul complexes of ${\mathcal L}$ over the ${\mathbb N}^d$-graded rings ${\mathcal R} \subseteq {\mathcal S}$
 and  the quotient complex. All these complexes are ${\mathbb N}^d$-graded and of length $r = r_1+\cdots+r_d$. We will show that, on the one hand, the Euler-Poincar\'{e} characteristic $\chi(K_{\bullet}(\phi_1,\cdots,\phi_d))$ equals 
the Euler-Poincar\'{e} characteristic of the complex of $R$-modules $K_{\bullet}(\LL; \s)_{n_1, \ldots, n_d}$, for $(n_1, \ldots, n_d) \geq (1,\ldots, 1)$, while on the other hand, the Euler-Poincar\'{e} characteristic of $K_{\bullet}(\LL; \s)_{n_1, \ldots, n_d}$
 equals $br (M_1|\cdots |M_d)$ for all $n_1, \ldots, n_d$ sufficiently large. 
 
 The proof of Theorem \ref{thm:brmultepchar} will follow from the three ensuing propositions.

 \begin{proposition}\label{prop:epphi} Let $(R,{\mathfrak m},k)$ be a Noetherian local ring of positive dimension $d$  and 
let \[B_1 \subseteq F_1, B_2 \subseteq F_2, \ldots,  B_d \subseteq F_d\] be 
$R$-submodules  of  free   $R$-modules $F_1,\cdots,F_d$ of ranks $r_1,\cdots,r_d$ such that $B_k$ is $r_k$-generated and such that $(det(B_1),\cdots,det(B_d))$ is an ideal of finite colength in $R$. For $1\leq k\leq d$, let $T_{k1},\cdots,T_{kr_k}$, $L_{k1},\cdots,L_{kr_k}$, ${\mathcal L}_k$, ${\mathcal L}$ and ${\mathcal S}$ be as above. Then the Euler-Poincar\'{e} characteristic of $K_{\bullet}({\mathcal L};{\mathcal S})_{(n_1,\cdots,n_d)}$ is finite and equals $\chi(K_{\bullet}(\phi_1,\cdots,\phi_d))$ for all $n_1,\cdots,n_d \geq 1$ (where $\phi_k$ is an endomorphism of $F_k$ with image $B_k$).
\end{proposition}

\begin{proof} Step I: We will first see that $H_{\bullet}({\mathcal L};{\mathcal S})_{(n_1,\cdots,n_d)}$ is of finite length for all $n_1,\cdots,n_d \geq 1$. By localising at a non-maximal prime, it suffices to see that $H_{\bullet}({\mathcal L};{\mathcal S})_{(n_1,\cdots,n_d)}$ vanishes for all $n_1,\cdots,n_d \geq 1$
if  $(det(B_1),\cdots, det(B_d)) = R$. The latter condition holds if and only if some $det(B_k)$ is a unit or equivalently some $B_k=F_k$.

Assume without loss of generality that $k=1$. By functoriality of Koszul homology, we may calculate $H_{\bullet}({\mathcal L};{\mathcal S})$ by calculating $H_{\bullet}({\mathcal L}^\prime;{\mathcal S})$ where ${\mathcal L}^\prime$ is obtained from ${\mathcal L}$ by replacing the generators $L_{11},\cdots,L_{1r_1}$ of $B_1$ by $T_{11},\cdots,T_{1r_1}$. We will now show that $H_{\bullet}({\mathcal L}^\prime;{\mathcal S})_{(n_1,\cdots,n_d)}$ vanishes whenever $n_1 \geq 1$ by showing 
\begin{enumerate}
\item[(1)] $H_{\bullet}(T_{11},\cdots,T_{1r_1};{\mathcal S})_{(n_1,\cdots,n_d)}$ vanishes whenever $n_1 \geq 1$, and
\item[(2)] If $H_{\bullet}({\mathcal F};{\mathcal S})_{(n_1,\cdots,n_d)}$ vanishes whenever $n_1 \geq 1$, then so does $H_{\bullet}({\mathcal F} \cup \{H\};{\mathcal S})_{(n_1,\cdots,n_d)}$  whenever $n_1 \geq 1$, where ${\mathcal F}$ is any set of linear forms and $H$ is a linear form of degree $e_k$ with $k \neq 1$.
\end{enumerate}
The statement (1) is clear since $T_{11},\cdots,T_{1r_1}$ is a regular sequence in ${\mathcal S}$ and so the only non-vanishing $H_{\bullet}(T_{11},\cdots,T_{1r_1};{\mathcal S})$ is $H_{0}(T_{11},\cdots,T_{1r_1};{\mathcal S})$ which is $\frac{\mathcal S}{(T_{11},\cdots,T_{1r_1})}$. The non-zero components of this all lie in degrees $(0,n_2,\cdots,n_d)$.

For (2) use the long exact sequence in Koszul homology:\begin{center}
\begin{tikzcd}
\cdots \rightarrow H_{1}({\mathcal F};{\mathcal S})(-e_k)    \arrow[r,"H"] & H_{1}({\mathcal F};{\mathcal S})  \arrow[r] & H_{1}({\mathcal F} \cup \{H\};{\mathcal S}) \arrow[lld]\\
 H_{0}({\mathcal F};{\mathcal S})(-e_k)    \arrow[r,"H"] & H_{0}({\mathcal F};{\mathcal S})  \arrow[r] & H_{0}({\mathcal F} \cup \{H\};{\mathcal S}) \rightarrow 0.
 \end{tikzcd}
\end{center}
which in degree $(n_1,\cdots,n_d)$ gives exactness of
 \begin{center}
\begin{tikzcd}
\cdots \rightarrow H_{1}({\mathcal F};{\mathcal S})_{(n_1,\cdots,n_k-1,\cdots,n_d)}    \arrow[r,"H"] & H_{1}({\mathcal F};{\mathcal S})_{(n_1,\cdots,n_k,\cdots,n_d)}  \arrow{ld}\\
H_{1}({\mathcal F} \cup \{H\};{\mathcal S})_{(n_1,\cdots,n_k,\cdots,n_d)} \arrow[r] & % \arrow[lld] &
H_{0}({\mathcal F};{\mathcal S})_{(n_1,\cdots,n_k-1,\cdots,n_d)} \arrow{ld}[swap]{H}\\%   \arrow{r}[swap]{"H"} \\
H_{0}({\mathcal F};{\mathcal S})_{(n_1,\cdots,n_k,\cdots,n_d)}   \arrow[r] & H_{0}({\mathcal F} \cup \{H\};{\mathcal S})_{(n_1,\cdots,n_k,\cdots,n_d)}  \rightarrow 0,
 \end{tikzcd}
 \end{center}

 thereby implying (2).
 
 Thus $H_{\bullet}({\mathcal L};{\mathcal S})_{(n_1,\cdots,n_d)}$ is of finite length for all $n_1,\cdots,n_d \geq 1$ and so the Euler-Poincar\'{e} characteristic 
of the homology of $K_{\bullet}({\mathcal L};{\mathcal S})_{(n_1,\cdots,n_d)}$ is finite for all $n_1,\cdots,n_d \geq 1$.

Step II:  Next we will see that the complex $K_{\bullet}({\mathcal L};{\mathcal S})_{(1,\cdots,1)}$ is isomorphic to $K_{\bullet}(\phi_1,\cdots,\phi_d)$ as complexes of (free) $R$-modules. By definition, $K_{\bullet}(\phi_1,\cdots,\phi_d)$ is the tensor product of the complexes $K_{\bullet}(\phi_k)$:
 \begin{center}
\begin{tikzcd}
 0  \arrow[r] &  F_k   \arrow[r,"\phi_k"]  & F_k   \arrow[r]  & 0 ~=~ 0  \arrow[r] &  R^{r_k}   \arrow[r,"B_k"]  & R^{r_k}   \arrow[r]  & 0
\end{tikzcd}
\end{center}
for $k=1,2,\cdots,d$.

On the other hand, as a complex of ${\mathbb N}^d$-graded $R$-modules, $K_{\bullet}({\mathcal L};{\mathcal S})$ is the tensor product (over $R$) of the complexes $K_{\bullet}({\mathcal L}_k;{\mathcal S}_k)$ where ${\mathcal S}_k = R[T_{k1},\cdots,T_{kr_k}]$ which is ${\mathbb N}^d$-graded with all the $T_{kj}$ having degree $e_k$ as also all the elements $L_{k1},\cdots,L_{kr_k}$ of ${\mathcal L}_k$.

Grading considerations now imply that for any ${n} = (n_1,\cdots,n_d) \in {\mathbb N}^d$, the degree ${n}$ component of the complex $K_{\bullet}({\mathcal L};{\mathcal S})$ is given by
$$
K_{\bullet}({\mathcal L};{\mathcal S})_{(n_1,\cdots,n_d)} = \bigotimes_{k=1}^d K_{\bullet}({\mathcal L}_k;{\mathcal S}_k)_{(0,\cdots,0,n_k,0,\cdots,0)}.
$$
In particular, 
$$
K_{\bullet}({\mathcal L};{\mathcal S})_{(1,\cdots,1)} = \bigotimes_{k=1}^d K_{\bullet}({\mathcal L}_k;{\mathcal S}_k)_{e_k}.
$$
The complex $K_{\bullet}({\mathcal L}_k;{\mathcal S}_k)$ has the form 
$$
0 \longrightarrow {\mathcal S}_k(-r_ke_k) \longrightarrow {\mathcal S}_k((-r_k+1)e_k)^{\binom{r_k}{r_k-1}} \longrightarrow \cdots \longrightarrow  {\mathcal S}_k(-e_k)^{\binom{r_k}{1}} \longrightarrow {\mathcal S}_k,
$$
and so $K_{\bullet}({\mathcal L}_k;{\mathcal S}_k)_{e_k}$ has the form
$$
0 \rightarrow   ({\mathcal S}_k)_0 ^{\binom{r_k}{1}} \rightarrow ({\mathcal S}_k)_{e_k}.
$$
In the monomial basis of ${\mathcal S}_k$, this last complex can be identified with
$$
0 \rightarrow R^{r_k} \stackrel{B_k}{\longrightarrow} R^{r_k} \rightarrow 0,
$$
finishing the proof of Step II.
 
 Step III:  Finally we will show that all the complexes $K_{\bullet}({\mathcal L};{\mathcal S})_{(n_1,\cdots,n_d)}$ have the same Euler-Poincar\'{e} characteristic for all $n_1,\cdots,n_d \geq 1$.  It suffices to see that for the injective map of complexes
 $$
 K_{\bullet}({\mathcal L};{\mathcal S}) \hookrightarrow K_{\bullet}({\mathcal L};{\mathcal S})(e_k)
 $$
 given by multiplication by some (any) $T_{kj}$, the quotient complex has finite length homology and vanishing Euler-Poincar\'{e} characteristic in every degree ${n} \geq (1,\cdots,1)$.

We may assume that $k=1$. Now, if $r_1 = 1$, then it is clear that multiplication by $T_{11}$ gives an isomorphism of complexes from $K_{\bullet}({\mathcal L};{\mathcal S})$ to $K_{\bullet}({\mathcal L};{\mathcal S})(e_1)$ and so the quotient complex actually vanishes - this is just a generalisation of the observation we made earlier that when all $r_k$ are 1, the complexes  $K_{\bullet}({\mathcal L};{\mathcal S})_{(n_1,\cdots,n_d)}$ are  all isomorphic. So we may assume that $r_1 > 1$ in the rest of the proof.

In this case, the quotient complex is $K_{\bullet}({\mathcal L};\frac{\mathcal S}{(T_{1j})}(e_1))$. It will suffice to see that $K_{\bullet}({\mathcal L};\frac{\mathcal S}{(T_{1j})})$ has finite length homology and vanishing Euler-Poincar\'{e} characteristic in every degree ${n} \geq (2,1,\cdots,1)$. We can regard this as Koszul homology over the ring $\tilde{\mathcal S} =\frac{\mathcal S}{(T_{1j})}$ of the image $\tilde{\mathcal L}$ of ${\mathcal L}$ in this ring. We're now done by the next proposition.
\end{proof}

\begin{proposition}\label{prop:oneless}
For $r_1>1$, let $$\tilde{\mathcal S} = R[T_{12},\cdots,T_{1r_1},T_{21},\cdots,T_{2r_2}, \cdots, T_{d1},\cdots,T_{dr_d}].$$ 
Let $B_1,\cdots,B_d$ be $r_1 \times r_1$, $\cdots$, $r_d \times r_d$ matrices such that $( (det(B_1),\cdots,det(B_d))$ is of finite colength in $R$. 
Let $\tilde{B}_1$ be the submatrix of $B_1$ obtained by deleting its first row.
Let $\tilde{L}_{ij}$ be linear forms in $\tilde{\mathcal S}$ determined by the columns of $\tilde{B}_1,B_2,\cdots,B_d$ and let $\tilde{\mathcal L}$ denote the collection of all these $\tilde{L}_{ij}$.
Then $K_{\bullet}(\tilde{\mathcal L};\tilde{\mathcal S})$ has finite length homology and vanishing Euler-Poincar\'{e} characteristic in every degree ${n} \geq (2,1,\cdots,1)$.
\end{proposition}

\begin{proof} By the functoriality of Koszul homology, we can perform column operations on $\tilde{B}_1$ without changing $K_{\bullet}(\tilde{\mathcal L};\tilde{\mathcal S})$.
We show that after certain such operations, if $C_1$ denotes the submatrix of $\tilde{B}_1$  comprising of the last $r_1-1$ columns, then $(det(C_1),det(B_2),\cdots,det(B_d))$ is of finite colength in $R$. To see this, 
 if $det(B_1)$ is a unit, then so is some maximal minor of $\tilde{B}_1$ and we can arrange for it to be the minor of the last $r_1-1$ columns. If any of $det(B_2),\cdots,det(B_d)$ is a unit, we need not do any column operations. If neither of these two conditions hold, then $(det(B_1),det(B_2),\cdots,det(B_d))$ is an ${\mathfrak m}$-primary ideal of $R$ and so, with $J = (det(B_2),\cdots,det(B_d))$, the dimension of $\frac{R}{J}$ is 1. Let $\{P_1,\cdots,P_n\}$ be the set of minimal primes of $J$. Then not all the maximal minors $\Delta_1,\cdots,\Delta_{r_1}$ of $\tilde{B}_1$ are contained in some $P_i$ - else so would $det(B_1)$. By Lemma \ref{lemma:avoidance}, there exist $a_2,\cdots,a_{r_1}$ such that $\Delta_1+a_2\Delta_2+\cdots+a_{r_1}\Delta_{r_1}$ is not contained in any $P_i$. Add $a_2$ times the first column of  $\tilde{B}_1$ to its second, $a_3$ times the first column  to its third and so on. For the modified matrix, if $C_1$ is the submatrix  of the last $r_1-1$ columns, then $det(C_1) = \Delta_1+a_2\Delta_2+\cdots+a_{r_1}\Delta_{r_1}$, which together with $det(B_2),\cdots,det(B_d)$ generates a finite colength ideal in $R$.
 
Let ${\mathcal L}^\prime = \tilde{\mathcal L} \setminus \{\tilde{L}_{11}\}$. By Step I of the previous proposition, $K_{\bullet}({\mathcal L}^\prime;\tilde{\mathcal S})$ has finite length homology  in every degree ${n} \geq (1,1,\cdots,1)$. Now the long exact sequence in Koszul homology gives the sequence
\begin{center}
\begin{tikzcd}
\cdots \rightarrow H_{1}({\mathcal L}^\prime;\tilde{\mathcal S})(-e_1)    \arrow[r,"L_{11}"] & H_{1}({\mathcal L}^\prime;\tilde{\mathcal S})  \arrow[r] & H_{1}(\tilde{\mathcal L};\tilde{\mathcal S}) \arrow[lld]\\
 H_{0}({\mathcal L}^\prime;\tilde{\mathcal S})(-e_1)    \arrow[r,"L_{11}"] & H_{0}({\mathcal L}^\prime;\tilde{\mathcal S})  \arrow[r] & H_{0}(\tilde{\mathcal L};\tilde{\mathcal S}) \rightarrow 0.
 \end{tikzcd}
\end{center}
which in degree $(n_1,\cdots,n_d)$ gives exactness of

\begin{center}
\begin{tikzcd}
\cdots \rightarrow H_{1}({\mathcal L}^\prime;\tilde{\mathcal S})_{(n_1-1,n_2,\cdots,\cdots,n_d)}    \arrow[r,"L_{11}"] & H_{1}({\mathcal L}^\prime;\tilde{\mathcal S})_{(n_1,\cdots,n_k,\cdots,n_d)}  \arrow[ld] \\
 H_{1}(\tilde{\mathcal L};\tilde{\mathcal S})_{(n_1,\cdots,n_k,\cdots,n_d)} \arrow[r] & %\arrow[lld] \\
 H_{0}({\mathcal L}^\prime;\tilde{\mathcal S})_{(n_1-1,n_2,\cdots,\cdots,n_d)}   \arrow{ld}[swap]{L_{11}} \\% \arrow[r,"L_{11}"] & \\
 H_{0}({\mathcal L}^\prime;\tilde{\mathcal S})_{(n_1,\cdots,n_k,\cdots,n_d)}  \arrow[r] & H_{0}(\tilde{\mathcal L};\tilde{\mathcal S})_{(n_1,\cdots,n_k,\cdots,n_d)} \rightarrow 0.
 \end{tikzcd}
\end{center}

Both the finiteness of the length of homology of $K_{\bullet}(\tilde{\mathcal L};\tilde{\mathcal S})$ and the vanishing of the Euler-Poincar\'{e} characteristic in degrees ${n} \geq (2,1,\cdots,1)$ follow from this long exact sequence.
\end{proof}

\medskip
\noindent
{\bf Example.} 
Let $R = k[[x,y]]$, $F_1 = R$, $F_2 = R^2$. Suppose that
$$
B_1 = \left[ x \right], B_2 = \left[ \begin{array}{cc} y & x \\ x & y \end{array} \right].
$$
A basis of $F_1$ is $\{T_{11}\}$ in degree $e_1 = (1,0)$ while a basis of $F_2$ is $\{T_{21},T_{22}\}$ in degree $e_2 = (0,1)$. Then, $$L_{11} = xT_{11}, L_{21} = yT_{21}+xT_{22}, L_{22}=xT_{21}+yT_{22}.$$
The ring ${\mathcal S} = R[T_{11},T_{21},T_{22}]$ which is naturally bigraded. The complex $K_{\bullet}({\mathcal L};{\mathcal S})$ is the Koszul complex over ${\mathcal S}$ with respect to the elements ${\mathcal L} = \{L_{11},L_{21},L_{22}\}$. This has the form
$$
0 \rightarrow {\mathcal S}(-1,-2) \stackrel{\partial_3}{\rightarrow} {\mathcal S}(-1,-1)^{\oplus 2} \oplus {\mathcal S}(0,-2) \stackrel{\partial_2}{\rightarrow} {\mathcal S}(-1,0) \oplus {\mathcal S}(0,-1)^{\oplus 2}  \stackrel{\partial_1}{\rightarrow} {\mathcal S} \rightarrow 0,
$$
with the matrices of the maps given by
$$
\partial_3 =
\left[ \begin{array}{r} L_{22} \\ -L_{21} \\ L_{11} \end{array} \right],
\partial_2 =
\left[ \begin{array}{rrr} L_{21}  & L_{22} & 0 \\ -L_{11} & 0 &  L_{22}\\ 0 & -L_{11} & -L_{21}  \end{array} \right],
\partial_1 =
\left[ \begin{array}{rrr} L_{11}  & L_{21} & L_{22}   \end{array} \right].
$$

The complex $K_{\bullet}({\mathcal L};{\mathcal S})_{(n_1,n_2)}$ is given by (with $n_i^\prime = n_i-1$, $n_i^{\prime\prime} = n_i-2$)
$$
0 \rightarrow {\mathcal S}_{(n_1^\prime,n_2^{\prime\prime})} \rightarrow {\mathcal S}_{(n_1^\prime,n_2^\prime)}^{\oplus 2} \oplus {\mathcal S}_{(n_1,n_2^{\prime\prime})} \rightarrow {\mathcal S}_{(n_1^\prime,n_2)} \oplus {\mathcal S}_{(n_1,n_2^\prime)}^{\oplus 2}  \rightarrow {\mathcal S}_{(n_1,n_2)} \rightarrow 0,
$$
with the restricted maps.
The $R$-module ${\mathcal S}_{(n_1,n_2)}$ is free of rank $n_2+1$ and has basis all $T_{11}^{n_1}T_{21}^kT_{22}^l$ where $k+l = n_2$. To specify matrices for maps between such modules, we will order this basis by decreasing $k$.
Then, the matrices for the maps $L_{11}: {\mathcal S}_{(n_1,n_2)} \rightarrow {\mathcal S}_{(n_1+1,n_2)}$ and $L_{21},L_{22}: {\mathcal S}_{(n_1,n_2)} \rightarrow {\mathcal S}_{(n_1,n_2+1)}$ are given by $L_{11} = xI_{n_2+1}$,
$$
 L_{21} = \left[
\begin{array}{cccccccc}
y & 0 & 0 & 0 & 0\\
x & y & 0 & 0 & 0\\
0 & x & y & 0 & 0\\
& & \ddots & \ddots & \\
0 & 0 & 0 & x & y\\
0 & 0 & 0 & 0 & x
\end{array}
\right]_{(n_2^{\prime\prime},n_2^\prime)}, 
L_{22} =  \left[
\begin{array}{cccccccc}
x & 0 & 0 & 0 & 0\\
y & x & 0 & 0 & 0\\
0 & y & x & 0 & 0\\
& & \ddots & \ddots & \\
0 & 0 & 0 & y & x\\
0 & 0 & 0 & 0 & y\\
\end{array}
\right]_{(n_2^{\prime\prime},n_2^\prime)}
$$

In particular, the complex $K_{\bullet}({\mathcal L};{\mathcal S})_{(1,1)}$ is given by
$$
0 \longrightarrow {\mathcal S}_{(0,0)}^{\oplus 2}  \stackrel{\partial_2}{\longrightarrow} {\mathcal S}_{(0,1)} \oplus {\mathcal S}_{(1,0)}^{\oplus 2}  \stackrel{\partial_1}{\longrightarrow} {\mathcal S}_{(1,1)} \longrightarrow 0
$$
with the matrices of these maps being
$$
\partial_2 =
\left[ \begin{array}{rr} y  & x \\ x & y \\ -x & 0 \\ 0 & -x   \end{array} \right], 
\partial_1 =
\left[ \begin{array}{rrrr} x  & 0 & y & x \\   0 &  x & x & y \end{array} \right].
$$
It is easy to see that this complex is exact with 0th homology of length 2, and hence has Euler-Poincar\'{e} characteristic 2.
As shown in Step II, this complex is isomorphic to the tensor product of the complexes $K_{\bullet}(B_1)$ and  $K_{\bullet}(B_2)$ which are given by
$$
0 \rightarrow R \rightarrow R \rightarrow 0
$$
$$
0 \rightarrow R^2 \rightarrow R^2 \rightarrow 0
$$
with maps given by the matrices $B_1$ and $B_2$.

The matrices for the general complex $K_{\bullet}({\mathcal L};{\mathcal S})_{(n_1,n_2)}$ (for $n_1 \geq 1$, $n_2 \geq 2$) are obtained by replacing $L_{11},L_{21},L_{22}$ in $\partial_1,\partial_2,\partial_3$ with the appropriate sized matrices described above. According to Proposition 16 all these complexes must have the same Euler-Poincar\'{e} characteristic 2 as $K_{\bullet}({\mathcal L};{\mathcal S})_{(1,1)}$.  It is clear from the matrices of $L_{11},L_{21},L_{22}$ that the complex $K_{\bullet}({\mathcal L};{\mathcal S})_{(n_1,n_2)}$ is independent of $n_1 \geq 1$ - this is because $B_1$ is of size 1, and corresponds to the $r_1=1$ case of Proposition 16.

To show the Euler-Poincar\'{e} characteristic of $K_{\bullet}({\mathcal L};{\mathcal S})_{(n_1,n_2)}$ is independent of $n_2 \geq 1$, we need to consider the quotient complex for the injective map of complexes
$$
 K_{\bullet}({\mathcal L};{\mathcal S}) \hookrightarrow K_{\bullet}({\mathcal L};{\mathcal S})(e_2)
 $$
 given by multiplication by any $T_{2j}$, and show that it has finite length homology and vanishing Euler-Poincar\'{e} characteristic in every degree ${n} \geq (1,2)$.

This is equivalent to seeing that with $\tilde{\mathcal S} = R[T_{11},T_{22}]$ and $$\tilde{\mathcal L} = \{\tilde{L}_{11} = xT_{11}, \tilde{L}_{21} =  xT_{22}, \tilde{L}_{22} = yT_{22}\},$$ $K_{\bullet}(\tilde{\mathcal L};\tilde{\mathcal S})$ has finite length homology and vanishing Euler-Poincar\'{e} characteristic in every degree ${n} \geq (1,2)$. Here, with the notation of Proposition 17, considering multiplication by $T_{21}$,
${\mathcal L}^\prime =  \{\tilde{L}_{11}, \tilde{L}_{22}\}$. These form a regular sequence in $\tilde{\mathcal S}$, so the homology of the Koszul complex on these is concentrated in degree 0 and equals $\frac{R[T_{11},T_{22}]}{(xT_{11},yT_{22})}$. This does have finite length homology, equal to $k$, in all degrees ${n} \geq (1,1)$, thereby finishing the proof using the long exact sequence in Koszul homology.
Note that the restriction on the degrees is actually necessary since for $n_1,n_2 \geq 1$,
$$
H_{\bullet}({\mathcal L}^\prime;\tilde{\mathcal S})_{(0,0)} \cong R, ~~H_{\bullet}({\mathcal L}^\prime;\tilde{\mathcal S})_{(n_1,0)} \cong \frac{R}{xR}, ~~H_{\bullet}({\mathcal L}^\prime;\tilde{\mathcal S})_{(0,n_2)} \cong \frac{R}{yR},
$$
none of which are of finite length.\qed

\begin{proposition}\label{prop:nlargemult} With notation as above and as in the statement of Theorem \ref{thm:brmultepchar}, 
$$
\chi(K_{\bullet}({\mathcal L};{\mathcal S})_{(n_1,\cdots,n_d)}) = br(M_1|\cdots|M_d)
$$
 for all sufficiently large $n_1,\cdots,n_d$.
\end{proposition}

\begin{proof} Observe first that as $(B_1,\cdots,B_d)$ is a joint reduction of $(M_1,\cdots,M_d)$, the joint reduction equation implies that the ideal generated by  ${\mathcal L}$ in ${\mathcal R} = S(M)$ is irrelevant  or equivalently that the degrees where the quotient ring $\frac{\mathcal R}{\mathcal L}$ is non-zero are given by the lattice points in a shaded region such as below in ${\mathbb N}^2$ (extended infinitely above and to the right) for $d=2$, and higher dimensional versions of this for larger $d$.
\begin{center}
\includegraphics[scale=0.4]{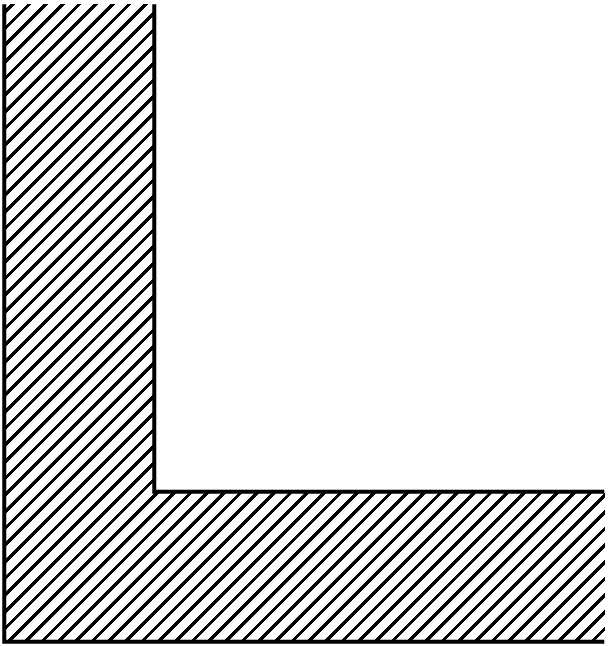}
\end{center}
The same is true for any finitely generated module over this quotient ring. In particular, the homology modules of the Koszul complex $K_{\bullet}({\mathcal L};{\mathcal R})$ all have that appearance. So $H_{\bullet}({\mathcal L};{\mathcal R})_{(n_1,\cdots,n_d)}$ vanishes for all $n_1, \cdots, n_d$ sufficiently large.

The inclusion of $K_{\bullet}({\mathcal L};{\mathcal R})$ in $K_{\bullet}({\mathcal L};{\mathcal S})$ looks like:
\begin{center}
\begin{tikzcd}
0 \arrow[r] &K_{r}({\mathcal L};{\mathcal R}) \arrow[r] \arrow[d] &\cdots \arrow[r] &K_{p}({\mathcal L};{\mathcal R}) \arrow[r]  \arrow[d]   &\cdots \arrow[r] &K_{0}({\mathcal L};{\mathcal R}) \arrow[r]   \arrow[d] &0\\
0 \arrow[r] &K_{r}({\mathcal L};{\mathcal S}) \arrow[r]  &\cdots \arrow[r] &K_{p}({\mathcal L};{\mathcal S}) \arrow[r]   &\cdots \arrow[r] &K_{0}({\mathcal L};{\mathcal S}) \arrow[r]  &0
\end{tikzcd}
\end{center}
Here each module is  ${\mathbb N}^d$-graded and explicitly, in position $p$,
\begin{eqnarray*}
K_{p}({\mathcal L};{\mathcal R}) &=& \bigoplus_{Z \subseteq {\mathcal L}, |Z|=p} {\mathcal R}(-{\nu(Z)}),\\
K_{p}({\mathcal L};{\mathcal S}) &=& \bigoplus_{Z \subseteq {\mathcal L}, |Z|=p} {\mathcal S}(-{\nu(Z)}),
\end{eqnarray*}
where, for a subset $Z$ of ${\mathcal L}$, the associated twist ${\nu(Z)} \in {\mathbb N}^d$ is defined to be the sum of the degrees of all elements in $Z$. Equivalently,  ${\nu(Z)} = (|Z \cap {\mathcal L}_1|,\cdots, |Z \cap {\mathcal L}_d|)$.

For a fixed ${n} = (n_1,\cdots,n_d) \in {\mathbb N}^d$, we consider the  terms of degree ${n}$ in the above diagram. Again, in position $p$, we get
\begin{eqnarray*}
K_{p}({\mathcal L};{\mathcal R})_{{n}} &=& \bigoplus_{Z \subseteq {\mathcal L}, |Z|=p} {\mathcal R}_{{n}-{\nu(Z)}} = \bigoplus_{Z \subseteq {\mathcal L}, |Z|=p} S_{{n}-{\nu(Z)}}(M),\\
K_{p}({\mathcal L};{\mathcal S})_{{n}} &=& \bigoplus_{Z \subseteq {\mathcal L}, |Z|=p} {\mathcal S}_{{n}-{\nu(Z)}} = \bigoplus_{Z \subseteq {\mathcal L}, |Z|=p} S_{{n}-{\nu(Z)}}(F).
\end{eqnarray*}

Thus, in the quotient complex, the module in position $p$ is given by
$$
 \bigoplus_{Z \subseteq {\mathcal L}, |Z|=p} \frac{S_{{n}-{\nu(Z)}}(F)}{S_{{n}-{\nu(Z)}}(M)},
$$
which is a finite length module with length given by
$$
 \sum_{Z \subseteq {\mathcal L}, |Z|=p} f({n}-{\nu(Z)}).
$$
The Euler-Poincar\'{e} characteristic of the quotient complex  is therefore given by
$$
 \sum_{p=0}^r (-1)^p \sum_{Z \subseteq {\mathcal L}, |Z|=p} f({n}-{\nu(Z)}) = \sum_{Z \subseteq {\mathcal L}} (-1)^{|Z|} f({n}-{\nu(Z)}).
$$

We claim that this equals
$$
(\Delta_1^{r_1}\Delta_2^{r_2} \cdots \Delta_d^{r_d} f) (n_1,\cdots,n_d)
$$
where the $\Delta_i$ are the (mutually commuting) difference operators defined by $$(\Delta_i f) (n_1,\cdots,n_d) = f(n_1,\cdots,n_d) - f(n_1,\cdots,n_i-1,\cdots,n_d).$$
To prove this claim, first note by induction that
\begin{eqnarray*}
(\Delta_1^{r_1}\Delta_2^{r_2} \cdots \Delta_d^{r_d} f) (n_1,\cdots,n_d) =&\\
 \sum_{p_1=0}^{r_1} \cdots \sum_{p_d=0}^{r_d} (-1)^{p_1+\cdots+p_d} &\binom{r_1}{p_1} \cdots \binom{r_d}{p_d} ~f(n_1-p_1,\cdots,n_d-p_d).
\end{eqnarray*}
Then finish by observing that for each $(p_1,\cdots,p_d) \leq (r_1,\cdots,r_d)$ the number of subsets $Z$ of ${\mathcal L}$ with ${\nu(Z)} = (p_1,\cdots,p_d)$ is given by $\binom{r_1}{p_1} \cdots \binom{r_d}{p_d}$ and each of these subsets contributes $(-1)^{p_1+\cdots+p_d} f(n_1-p_1,\cdots,n_d-p_d)$ to the Euler-Poincar\'{e} characteristic.

We have shown that the Euler-Poincar\'{e} characteristic of the quotient complex (as well as that of its homology) in degree $(n_1,\cdots,n_d)$ is given by 
$$
(\Delta_1^{r_1}\Delta_2^{r_2} \cdots \Delta_d^{r_d} f) (n_1,\cdots,n_d).
$$
Since $f$ is given by a polynomial  $p(n_1,\cdots,n_d)$ for all sufficiently large $n_1,\cdots,n_d$ with $deg(p) = r = r_1+\cdots+r_d$, it follows that the Euler-Poincar\'{e} characteristic of the quotient complex  in degree $(n_1,\cdots,n_d)$ is given by $r_1!\cdots r_d!$ times the coefficient of $n_1^{r_1}\cdots n_d^{r_d}$ in $p$, for all 
sufficiently large $n_1,\cdots,n_d$. Therefore, the  Euler-Poincar\'{e} characteristic of the quotient complex  in degree $(n_1,\cdots,n_d)$ is the mixed Buchsbaum-Rim multiplicity $br(M_1|\cdots|M_d)$ for all sufficiently large $n_1,\cdots,n_d$.

We have already seen that $K_{\bullet}({\mathcal L};{\mathcal R})_{(n_1,\cdots,n_d)}$ is exact for all sufficiently large $n_1,\cdots,n_d$. It follows that the  Euler-Poincar\'{e} characteristic of $K_{\bullet}({\mathcal L};{\mathcal S})_{(n_1,\cdots,n_d)}$ is $br(M_1|\cdots|M_d)$ for all sufficiently large $n_1,\cdots,n_d$.
But by Proposition \ref{prop:epphi}, the  Euler-Poincar\'{e} characteristic of  $K_{\bullet}({\mathcal L};{\mathcal S})_{(n_1,\cdots,n_d)}$ is the same for all $n_1,\cdots,n_d \geq 1$ and equals $\chi(K_{\bullet}(\phi_1,\cdots,\phi_d))$.
\end{proof} 

Note that the complex $K_{\bullet}({\mathcal L};{\mathcal S})$ does not depend directly on $M_1,\cdots,M_d$ but only on the joint reduction $(B_1,\cdots,B_d)$. Also note that when all $r_k=1$, so that $M_1,\cdots,M_d$ is just a collection  of ${\mathfrak m}$-primary ideals and the joint reduction is just a collection $b_1,\cdots,b_d$ of elements of $R$, then, not just  is the  Euler-Poincar\'{e} characteristic of  $K_{\bullet}({\mathcal L};{\mathcal S})_{(n_1,\cdots,n_d)}$ the same for  all $n_1,\cdots,n_d \geq 1$,  the complexes  $K_{\bullet}({\mathcal L};{\mathcal S})_{(n_1,\cdots,n_d)}$ are themselves all isomorphic to the Koszul complex of 
$b_1,\cdots,b_d$, thereby proving Theorem \ref{thm:brmultepchar} in this case, without any need of Propositions \ref{prop:epphi} and \ref{prop:oneless}. 

\medskip
\noindent
\textit{Proof of Theorem \ref{thm:brmultepchar}.} By Proposition \ref{prop:epphi},   $\chi(K_{\bullet}(\phi_1,\cdots,\phi_d))$ equals the Euler-Poincar\'{e} characteristic of  $K_{\bullet}({\mathcal L};{\mathcal S})_{(n_1,\cdots,n_d)}$ for all $n_1,\cdots,n_d \geq 1$ and by Proposition \ref{prop:nlargemult}, the Euler-Poincar\'{e} characteristic of the complex $K_{\bullet}({\mathcal L};{\mathcal S})_{(n_1,\cdots,n_d)}$ equals $br(M_1|\cdots|M_d)$ for all sufficiently large $n_1,\cdots,n_d$.
\qed

\section{Mixed Buchsbaum-Rim multiplicities of modules and ideals}

Suppose, as in the previous section, $M_1\subseteq F_1, \ldots, M_d\subseteq F_d$ is a collection of finite $R$-modules such that each $M_k$ has finite colength in $F_k$. The purpose of this section is to prove that 
$br(M_1|\cdots | M_d) = e(I_1|\cdots | I_d)$, the mixed multiplicity of $I_1, \ldots, I_d$, where $I_k$ denotes the ideal of maximal minors of the matrix whose columns are the generators of $M_k$. However, most of the work done in this section goes towards generalising the following lemma stated slightly differently in \cite{Grt1967} and \cite{Flt1998}. We give both statements below. We believe the generalisation we provide is interesting in its own right. 

\begin{lemma}[Lemma 21.10.17.3 of \cite{Grt1967}]
Let $R$ be a Noetherian local ring of dimension 1, $F$ be a free $R$-module of finite rank and $\phi: F \rightarrow F$ be an injective endomorphism.  Then
$$
\lambda\left( \frac{F}{im(\phi)} \right) = \lambda\left( \frac{R}{(det(\phi))} \right).
$$
\end{lemma}

Over a Noetherian ring $R$, for an endomorphism $\phi$ of a finitely generated module $M$,  say that $e(\phi,M)$ is defined if both $ker(\phi)$ and $coker(\phi)$ are finite length $R$-modules and in this case, set $e(\phi,M) = \lambda(coker(\phi)) - \lambda(ker(\phi))$.

\begin{lemma}[Lemma A.2.6 of \cite{Flt1998}] Let $R$ be a Noetherian ring, $F$ be a free $R$-module of finite rank and $\phi: F \rightarrow F$ be an endomorphism. Then $e(\phi,F)$ is defined if and only if $e(det(\phi),R)$ is defined, and when they are both defined, they are equal.
\end{lemma}

Fulton refers to this as an ``important" lemma and gives an elegant  proof motivated by K-theory that he attributes to Iversen. Grothendieck's proof proceeds by completing and reducing to the case of a discrete valuation ring. Our proof of the generalisation is based on Fulton's. 

We begin with free submodules $F_1, \ldots, F_q$, with $F_k$ of rank $r_k$. For each $1\leq k\leq q$, let $\phi_k: F_k \rightarrow F_k$ be an endomorphism, and, as before, set $K_{\bullet}(\phi_1, \ldots, \phi_q)= K(\phi_1)\otimes \cdots \otimes K(\phi_q)$, where $K_{\bullet}(\phi_k)$ is the complex:
 \begin{center}
\begin{tikzcd}
 0  \arrow[r] &  F_k   \arrow[r,"\phi_k"]  & F_k   \arrow[r]  & 0,
\end{tikzcd}
\end{center} As a complex of $R$-modules, one can can describe $K_{\bullet}(\phi_1, \ldots, \phi_q)$ as follows: Set $F = F_1\otimes \cdots \otimes F_q$
 \[\Phi_p = 1_{F_1}\otimes\cdots  1_{F_{p-1}}\otimes \phi_p\otimes 1_{F_{p+1}}\otimes \cdots \otimes 1_{F_q}, \] for each $0\leq p\leq q$, and write $C_p$ for set set of all set  of $p$-tuples  $(i_1, \ldots, i_p)$ with $0\leq i_1< \cdots <  i_p\leq q$ and if $C\in C_p$, write $C(i_j)$ for the element of $C_{p-1}$ with $i_j$ removed from $C$. Set $K_p = \bigoplus_{C\in C_p} F_C$, where $F_C$ is just a copy of $F$. Then it is not difficult to see that $K_{\bullet}(\phi_1, \ldots, \phi _d)$ is the complex of $R$-modules
 \[
 0\lra K_q\lra K_{q-1}\cdots \lra K_p\overset{\partial _p}\lra K_p\cdots \lra K_1\to K_0\lra 0,
 \] where $\partial_p (v_C) = \sum_{j = 1}^p (-1)^{j+1}\Phi_j(v_{C(i_j)})$, for $C\in C_p$. This suggests that we can treat $F$ as a module over $R[X_1, \ldots, X_q]$ by letting $X_k$ act as $\Phi_k$. Thus, as a complex of $R[X_1, \ldots, X_q]$-modules, 
\[K_{\bullet}(\phi_1, \ldots, \phi_d) = K_{\bullet}(X_1, \ldots, X_q; F),
 \] the Koszul complex of $X_1, \ldots, X_q$ on $F$. In the proof below, we will have occasion to change modules and endomorphisms, so we will write $F(\phi_1, \ldots, \phi_d) = F$, to indicate the dependence of $F$ on the $F_k$ and $\phi_k$. Thus, if $G$ is another free $R$-module with endomorphism $\psi$, we will write $F(\psi, \phi_2, \ldots, \phi _q)$ for $G\otimes F_2\otimes \cdots \otimes F_q$ with the understanding that $X_1$ acts as $\psi\otimes 1_{F_2}\otimes \cdots \otimes 1_{F_q}$, and the remaining $X_k$ act in the obvious manner, using $1_G$ instead of $1_{F_1}$. 
 
 \medskip
Our goal is to compare the complex $K_{\bullet}(\phi_1, \ldots, \phi_q)$ with homology $H_{\bullet}(\phi_1, \ldots, \phi _q)$ to the usual Koszul complex $K_{\bullet}(det(\phi_1),\cdots,det(\phi_q))$ over $R$ and its homology denoted by $H_{\bullet}(det(\phi_1),\cdots,det(\phi_q))$. Our generalisation of the lemma above is the following.

\begin{theorem}{\label {thm:comparison}}
With notation as above, the homology of $K_{\bullet}(\phi_1,\cdots,\phi_q)$ is of finite length (as an $R$-module) if and only if $(det(\phi_1),\cdots,det(\phi_q))$ is a finite colength ideal of $R$, in which case, there is an equality of Euler-Poincar\'{e} characteristics
$$
\chi(K_{\bullet}(\phi_1,\cdots,\phi_q)) = \chi(K_{\bullet}(det(\phi_1),\cdots,det(\phi_q)).
$$
Further, if $R$ is Cohen-Macaulay and $q=d$, both the Koszul complexes are acyclic and hence,
$$
\lambda\left( H_0(\phi_1, \ldots, \phi _d)\right) = \lambda\left( \frac{R}{(det(\phi_1),\cdots,det(\phi_d))} \right).
$$
\end{theorem}

Before embarking on the proof of this theorem, we will prove a few preliminary results that we need. We omit the proof of the next lemma which follows directly from the version of prime avoidance stated in Lemma 1.2.2 of \cite{BrnHrz1993}.

\begin{lemma}\label{lemma:avoidance}
Let $R$ be a ring, $\{P_1,\cdots,P_n\}$ be a finite set of prime ideals of $R$ 
 and for some $r \in {\mathbb N}$, let $b_1,\cdots,b_r \in R$ be such that they are not all contained in any $P_i$. Then there exist
$a_2,\cdots,a_r \in R$ such that $b_1+a_2b_2+\cdots+a_rb_r$ is not contained in any $P_i$. \qed
\end{lemma}
 
By an elementary matrix over $R$, we will mean a matrix of the form $I+aE(i,j)$ where $I$ is the identity matrix, $a \in R$ and $i \neq j$. Multiplication by an elementary matrix on the left or right of any matrix is equivalent to adding a multiple of a row or column to another. By a triangular matrix over $R$, we will mean a matrix that is either upper or lower triangular (not strictly). Elementary matrices are triangular, as are diagonal matrices.

\begin{lemma}\label{lemma:11entry}
Let $B$ be an $r \times r$ matrix over $R$ such that $det(B)$ is not contained in any  $P_i$ for a finite set of prime ideals $\{P_1,\cdots,P_n\}$ of $R$. Then $B$ can be written as  $B^\prime Q$ where $b^\prime_{11}$ is not in any $P_i$ and $Q$ is a product of elementary matrices.
\end{lemma}

\begin{proof}
Since $det(B)$ is not contained in any $P_i$, the entries $b_{11},\cdots,b_{1r}$ of the first row of $B$ are not all contained in any $P_i$.
By Lemma~\ref{lemma:avoidance}, there exist $a_2,\cdots,a_r \in R$ such that $b_{11}+a_2b_{12}+\cdots+a_rb_{1r}$ is not contained in any $P_i$. Let $B^\prime$ be obtained from $B$ by adding $a_2$ times its second column to the first, $a_3$ times its third column to the first, and so on up to $a_r$ times its $r$th column to its first. Then $b^\prime_{11} = b_{11}+a_2b_{12}+\cdots+a_rb_{1r}$ is not contained in any $P_i$ and $B$ is of the form $B^\prime Q$ where $Q$ is a product of elementary matrices.
\end{proof}

\begin{proposition}\label{prop:trian}
Let $B$ be an $r \times r$ matrix over $R$ such that $det(B)$ is not contained in any  $P_i$ for a finite set of prime ideals $\{P_1,\cdots,P_n\}$ of $R$. Then there exists $b \in R$ that is not in any $P_i$ such that $bB$ is a product of triangular matrices.
\end{proposition}

\begin{proof} We will prove this by induction on $r$, the basis case $r=1$ being clear. For $r>1$, by Lemma \ref{lemma:11entry}, we may assume that $b_{11}$ is not contained in any $P_i$. 
It is then easy to check that $b_{11}^rB$ can be written as
$$
\left[
\begin{array}{ccccc}
b_{11}^{r-1} & 0 & 0 & \cdots & 0\\
b_{11}^{r-2}b_{21} & b_{11}^{r-1} & 0 & \cdots & 0\\
b_{11}^{r-2}b_{31} & 0 & b_{11}^{r-1} & \cdots & 0\\
\vdots & \vdots & \vdots &\ddots & \vdots\\
b_{11}^{r-2}b_{r1} & 0 & 0 & \cdots & b_{11}^{r-1}
\end{array}
\right]
\left[
\begin{array}{ccccc}
b_{11}^2 & b_{11}b_{12} & b_{11}b_{13} & \cdots & b_{11}b_{1r}\\
0 & \Delta_{22} &  \Delta_{23} & \cdots &  \Delta_{2r}\\
0 &  \Delta_{32} &  \Delta_{33} & \cdots &  \Delta_{3r}\\
\vdots & \vdots & \vdots &\ddots & \vdots\\
0 &  \Delta_{r2} &  \Delta_{r3} & \cdots &  \Delta_{rr}
\end{array}
\right],
$$
where $\Delta_{ij} = b_{11}b_{ij}-b_{1j}b_{i1}$. 
It now suffices to write a multiple of the second matrix by some element not in any $P_i$ as a product of triangular matrices.
With $C$ denoting the matrix
$$
\left[
\begin{array}{cccc}
\Delta_{22} &  \Delta_{23} & \cdots &  \Delta_{2r}\\
\Delta_{32} &  \Delta_{33} & \cdots &  \Delta_{3r}\\
 \vdots & \vdots &\ddots & \vdots\\
 \Delta_{r2} &  \Delta_{r3} & \cdots &  \Delta_{rr}
\end{array}
\right],
$$
$det(C) = b_{11}^{r-2} det(B)$ and is not contained in any $P_i$. By induction, there is a $c \in R$ not contained in any $P_i$ such that $cC$ is a product of triangular matrices. Hence,
\begin{eqnarray*}
c
\left[
\begin{array}{cccc}
b_{11}^2 & b_{11}b_{12} &  \cdots & b_{11}b_{1r}\\
0 & \Delta_{22}  & \cdots &  \Delta_{2r}\\
\vdots & \vdots &\ddots & \vdots\\
0 &  \Delta_{r2} &  \cdots &  \Delta_{rr}
\end{array}
\right] &=&
\left[
\begin{array}{ccccc}
1 & 0 &  \cdots & 0\\
0 & c\Delta_{22} &   \cdots &  c\Delta_{2r}\\
\vdots & \vdots &\ddots & \vdots\\
0 &  c\Delta_{r2} &  \cdots &  c\Delta_{rr}
\end{array}
\right]\\
&\times&
\left[
\begin{array}{ccccc}
cb_{11}^2 & cb_{11}b_{12} &  \cdots & cb_{11}b_{1r}\\
0 & 1 &   \cdots &  0\\
\vdots & \vdots &\ddots & \vdots\\
0 &  0 &  \cdots & 1
\end{array}
\right],
\end{eqnarray*}
which is a product of triangular matrices.
\end{proof}

\begin{proof}[Proof of Theorem \ref{thm:comparison}]  
Step I: First we will see that $H_{0}(\phi_1,\cdots,\phi_q) $ vanishes exactly when $H_{0}(det(\phi_1),\cdots,det(\phi_q))$ does. By definition of $K_{\bullet}(\phi_1,\cdots,\phi_q)$, 
\begin{eqnarray*}
H_{0}(\phi_1,\cdots,\phi_q) &=& \frac{F(\phi_1,\cdots,\phi_q)}{(X_1,\cdots, X_q)F(\phi_1,\cdots,\phi_q)}\\
&=& \frac{F_1}{im(\phi_1)} \otimes \cdots \otimes \frac{F_q}{im(\phi_q)}
\end{eqnarray*}
The second equality above is of $R$-modules and holds because, denoting the image of $\phi_k$ by $B_k$, it follows that $X_kF(\phi_1,\cdots,\phi_q)$ is the submodule
$$F_1 \otimes \cdots F_{k-1} \otimes B_k \otimes F_{k+1} \otimes \cdots \otimes F_q$$ of $F(\phi_1,\cdots,\phi_q)$. Hence $(X_1,\cdots, X_q)F(\phi_1,\cdots,\phi_q)$ is the 
sum
\begin{eqnarray*}
(B_1 \otimes F_2 \otimes \cdots \otimes F_q)+ &\cdots& +(F_1 \otimes \cdots \otimes F_{k-1} \otimes B_k \otimes F_{k+1} \otimes \cdots \otimes F_q)\\
 + &\cdots& + \ (F_1 \otimes F_2 \otimes \cdots \otimes F_{q-1} \otimes B_q).
\end{eqnarray*}
Now use that for free $F$ and $G$,
$$
\frac{F}{B} \otimes \frac{G}{C} = \frac{F \otimes G}{B \otimes G + F \otimes C},
$$
and induction ($B \otimes G$ and  $F \otimes C$ are submodules of $F \otimes G$ by freeness).
Thus $H_{0}(\phi_1,\cdots,\phi_q) $ vanishes exactly when some some $\phi_k$ is surjective.
Similarly,
\begin{eqnarray*}
 H_{0}(det(\phi_1),\cdots,det(\phi_q)) &=& \frac{R}{(det(\phi_1),\cdots,det(\phi_q))},
\end{eqnarray*}
and therefore vanishes exactly when some  $det(\phi_k)$ is a unit. 
But, surjectivity of $\phi_k$ and $det(\phi_k)$ being a unit are equivalent conditions. 

Next, suppose that some $\phi_k$ is surjective (and therefore an isomorphism, being a surjective endomorphism of a Noetherian module). Since $K_{\bullet}(\phi_k)$ is a tensor factor of the complex $K_{\bullet}(\phi_1,\cdots,\phi_q)$, it follows easily that all the homology $H_{\bullet}(\phi_1,\cdots,\phi_q)$ vanishes.

Applying  these two observations to the localisations of the two Koszul complexes at every non-maximal prime of $R$, we conclude that $H_{\bullet}(\phi_1,\cdots,\phi_q)$ has finite length if and only if  
$H_{0}(det(\phi_1),\cdots,det(\phi_q))$ does. For the rest of the proof, we will assume that these equivalent conditions hold.

Step II: Let $X$ be the set of all $(\phi_1,\cdots,\phi_q) \in End(F_1) \times \cdots \times End(F_q)$ such that $det(\phi_1),\cdots,det(\phi_q)$ generate an ideal of finite colength in $R$, and let $Y$ be the subset of $X$ of all  $(\phi_1,\cdots,\phi_q)$ such that 
$$
\chi(K_{\bullet}(\phi_1,\cdots,\phi_q)) = \chi(K_{\bullet}(det(\phi_1),\cdots,det(\phi_q)).
$$
 We need to see that $Y=X$ (equivalently that $X \subseteq Y$). We will show the following.
\begin{enumerate}
\item Suppose $(\phi_1,\cdots,\phi_q) \in X$ and that for each $k = 1,2 \cdots, q$, with respect to some basis of $F_k$, the matrix of $\phi_k$ is triangular (we will refer to such $\phi_k$ as triangularisable). Then, $(\phi_1,\cdots,\phi_q) \in Y$.
\item For any $k$, if $\phi_k=\gamma_k\rho_k$ for endomorphisms $\gamma_k,\rho_k$ of $F_k$, and if any two of 
\begin{eqnarray*}
(\phi_1,\cdots,\phi_{k-1}, \phi_k,\phi_{k+1}, \cdots,\phi_d),\\ (\phi_1,\cdots,\phi_{k-1}, \gamma_k,\phi_{k+1}, \cdots,\phi_d),\\(\phi_1,\cdots,\phi_{k-1}, \rho_k,\phi_{k+1}, \cdots,\phi_d),
\end{eqnarray*}
 are in $Y$, then so is the third.
\item If $(\phi_1,\cdots,\phi_q) \in X$ and $1 \leq k \leq q$, then there exists $a_k \in R$ such that $$(\phi_1,\cdots,\phi_{k-1}, a_k\phi_k,\phi_{k+1}, \cdots,\phi_q) \in X$$ and such that $a_k\phi_k$ is a product of triangularisable endomorphisms.
\end{enumerate}

Given (1), (2) and (3), we will finish the proof. Suppose $(\phi_1,\cdots,\phi_q) \in X$. We show by induction on the number of non-triangularisable $\phi_k$'s that $(\phi_1,\cdots,\phi_q) \in Y$. The basis case when this number is $0$ is just (1). Suppose we have shown that whenever there are fewer than $k$ non-triangularisable $\phi_k$'s that $(\phi_1,\cdots,\phi_q) \in Y$. Take a $(\phi_1,\cdots,\phi_q) \in X$ with $k$ non-triangularisable $\phi_k$'s. Suppose, without loss of generality, that $\phi_1$ is non-triangularisable. By (3), find $a_1$ such that $(a_1\phi_1,\phi_2,\cdots,\phi_q) \in X$ and $a_1\phi_1$ being a product of triangularisable endomorphisms. For each of these triangularisable endomorphisms $\tau$, note that $(\tau,\phi_2,\cdots,\phi_q) \in X$  and hence by the inductive hypothesis, $(\tau,\phi_2,\cdots,\phi_q) \in Y$. By (2),
we see that $(a_1\phi_1,\phi_2,\cdots,\phi_q) \in Y$. Hence $(a_11_{F_1},\phi_2,\cdots,\phi_q) \in X$ and therefore again by the inductive hypothesis, $(a_11_{F_1},\phi_2,\cdots,\phi_q) \in Y$. By (2) again, we conclude that $(\phi_1,\cdots,\phi_q) \in Y$ as needed. We will now prove (1), (2) and (3).

Step III: Proof of (1): The proof is by induction on the sum of ranks $r_1+\cdots+r_d$. If this sum is $d$, then the statement is a tautology. For the inductive step, assume without loss of generality that $r_1 \geq 2$. Choose a basis of $F_1$ with respect to which the matrix of $\phi_1$ is upper triangular. 
Take the direct sum decomposition of $F_1$ as $G_1 \oplus H_1$ where $G_1$ is spanned by the first basis element and $H_1$ by the rest. 
Let $\psi_1$ and $\theta_1$ be the compressions of $\psi_1$ to $G_1$ and $H_1$, so that they are given by the composite maps
\begin{center}
\begin{tikzcd}
 G_1 \arrow[r]  &  G_1 \oplus H_1   \arrow[r,"\phi_1"]  & G_1 \oplus H_1   \arrow[r]  &G_1, \\
H_1 \arrow[r] &  G_1 \oplus H_1    \arrow[r,"\phi_1"] & G_1 \oplus H_1   \arrow[r]  & H_1.
\end{tikzcd}
\end{center}
The $R[X_1,\cdots,X_q]$-modules 
\begin{eqnarray*}
F(\psi_1,\phi_2,\phi_3,\cdots,\phi_q) &=& G_1 \otimes F_2 \otimes \cdots \otimes F_q{\text {~~\ and}}\\
 F(\theta_1,\phi_2,\phi_3,\cdots,\phi_q) &=& H_1 \otimes F_2 \otimes \cdots \otimes F_q
 \end{eqnarray*}
fit in a short exact sequence of $R[X_1,\cdots,X_q]$-modules
$$
0 \rightarrow F(\psi_1,\phi_2,\phi_3,\cdots,\phi_q) \rightarrow F(\phi_1,\cdots,\phi_q) \rightarrow F(\theta_1,\phi_2,\phi_3,\cdots,\phi_q) \rightarrow 0,
$$
where the first map is induced by the inclusion of $G_1$ into $F_1$ and the second by the projection of $F_1$ onto $H_1$.
Tensoring this (over $R[X_1,\cdots,X_q]$) with the Koszul complex of $(X_1,\cdots,X_q)$ over $R[X_1,\cdots,X_q]$ gives a short exact sequence of chain complexes
$$
0 \rightarrow K_{\bullet}(\psi_1,\phi_2,\phi_3,\cdots,\phi_q) \rightarrow K_{\bullet}(\phi_1,\cdots,\phi_q)  \rightarrow K_{\bullet}(\theta_1,\phi_2,\phi_3,\cdots,\phi_q) \rightarrow 0.
$$

Since  $(\phi_1,\cdots,\phi_q) \in X$ and $det(\phi_1) = det(\psi_1)det(\theta_1)$, both $(\psi_1,\phi_2,\phi_3,\cdots,\phi_q)$ and $(\theta_1,\phi_2,\phi_3,\cdots,\phi_q)$
are also in $X$. Further $\psi_1$ and $\theta_1$ are triangular with respect to the chosen bases of $G_1$ and $H_1$. So by induction on $r_1+\cdots+r_d$, both $(\psi_1,\phi_2,\phi_3,\cdots,\phi_q)$ and $(\theta_1,\phi_2,\phi_3,\cdots,\phi_q)$
are  in $Y$ and so 
\begin{eqnarray*}
\chi(K_{\bullet}(\psi_1,\phi_2,\phi_3,\cdots,\phi_q)) &=& \chi(K_{\bullet}(det(\psi_1),det(\phi_2),det(\phi_3),\cdots,det(\phi_q)),\\
\chi(K_{\bullet}(\theta_1,\phi_2,\phi_3,\cdots,\phi_q)) &=& \chi(K_{\bullet}(det(\theta_1),det(\phi_2),det(\phi_3),\cdots,det(\phi_q)).
\end{eqnarray*}
It follows from the short exact sequence of complexes that
\begin{eqnarray*}
\chi(K_{\bullet}(\phi_1,\cdots,\phi_q)) &=& \chi(K_{\bullet}(det(\psi_1),det(\phi_2),det(\phi_3),\cdots,det(\phi_q))\\
 &+& \chi(K_{\bullet}(det(\theta_1),det(\phi_2),det(\phi_3),\cdots,det(\phi_q)).
\end{eqnarray*}

Applying the above to the case when all $F_k$ are $R$ and the endomorphisms $\psi_1,\theta_1,\phi_1,\cdots, \phi_q$ are replaced by their determinants, we see that
\begin{eqnarray*}
\chi(K_{\bullet}(det(\phi_1),\cdots,det(\phi_q))) &=& \chi(K_{\bullet}(det(\psi_1),det(\phi_2),det(\phi_3),\cdots,det(\phi_q))\\
 &+& \chi(K_{\bullet}(det(\theta_1),det(\phi_2),det(\phi_3),\cdots,det(\phi_q)).
\end{eqnarray*}

Comparing the previous two equalities shows that $(\phi_1,\cdots,\phi_q) \in Y$.

Step IV: Proof of (2): The proof of this step is motivated by the proof of Proposition 5.2.3 in \cite{Rbr1998}. We may assume without loss of generality that $k=1$. Begin by noting that if any two of $(\phi_1,\phi_2,\cdots,\phi_q), (\gamma_1,\phi_2,\cdots,\phi_q)$ and $(\rho_1,\phi_2,\cdots,\phi_q)$
 are in $X$, then so is the third. Since some two are in $Y$, all three are in $X$.

It will suffice to see that
$$
\chi(K_{\bullet}(\phi_1,\phi_2,\cdots,\phi_q)) = \chi(K_{\bullet}(\gamma_1,\phi_2,\cdots,\phi_q))+\chi(K_{\bullet}(\rho_1,\phi_2,\cdots,\phi_q))
$$ for then, for the same reasons,
\begin{eqnarray*}
\chi(K_{\bullet}(det(\phi_1),det(\phi_2),\cdots,det(\phi_q))) &=& \chi(K_{\bullet}(det(\gamma_1),det(\phi_2),\cdots,det(\phi_q)))\\
&+& \chi(K_{\bullet}(det(\rho_1),det(\phi_2),\cdots,det(\phi_q))).
\end{eqnarray*}
Then, equality of any two of the three pairs of corresponding terms in the equations above implies the equality of the third pair, as needed.

Consider the endomorphism $\phi$ of $F_1 \oplus F_1$ given by 
$$
\phi =
\left[
\begin{array}{cc}
\rho_1 & 1_{F_1}\\
0 & \gamma_1
\end{array}
\right]
$$
and the $R[X_1,\cdots,X_q]$-module $F(\phi,\phi_2,\cdots,\phi_q)$, which, as an $R$-module is $$(F_1 \oplus F_1) \otimes F_2 \otimes \cdots \otimes F_q.$$ We claim that this fits into 
a short exact sequence of $R[X_1,\cdots,X_q]$-modules given by:
\begin{center}
\small{
\begin{tikzcd}
 0  \arrow[r]  & F(\rho_1,\phi_2,\cdots,\phi_q)  \arrow[r,"\spmat{ 1_{F_1}\\0 }"]  &  F(\phi,\phi_2,\cdots,\phi_q)   \arrow[r,"\spmat{ 0 & 1_{F_1} }"]  & F(\gamma_1,\phi_2,\cdots,\phi_q)    \arrow[r]  & 0,
\end{tikzcd}
}
\end{center}
where the matrices above the arrows indicate maps $F_1 \rightarrow F_1 \oplus F_1$ and $F_1 \oplus F_1 \rightarrow F_1$. The actual maps are the tensor product of these with the identity maps on the rest of the factors of the tensor products. While it is clear that this is a short exact sequence of $R$-modules, that the maps involved are $R[X_1,\cdots,X_q]$-module maps reduces quickly to  the commutativity of the diagram below.

\begin{center}
\begin{tikzcd}
  F_1 \arrow[r,"\spmat{ 1_{F_1}\\0 }"]  \arrow[d,"\rho_1"'] &  (F_1 \oplus F_1)  \arrow[r,"\spmat{ 0 & 1_{F_1} }"] \arrow[d,"\spmat{ \rho_1 & 1_{F_1}\\0 & \gamma_1 }"] & F_1 \arrow[d,"\gamma_1"],\\
 F_1 \arrow[r,"\spmat{ 1_{F_1}\\0 }"']  &  (F_1 \oplus F_1)  \arrow[r,"\spmat{ 0 & 1_{F_1} }"']  & F_1
\end{tikzcd}
\end{center}

Tensoring the short exact sequence of $R[X_1,\cdots,X_q]$-modules above by the Koszul complex $K_{\bullet}(X_1,\cdots,X_q)$ (over $R[X_1,\cdots,X_q]$) gives the short exact sequence of complexes
$$
0 \rightarrow K_{\bullet}(\rho_1,\phi_2,\phi_3,\cdots,\phi_q) \rightarrow K_{\bullet}(\phi,\phi_2\cdots,\phi_q)  \rightarrow K_{\bullet}(\gamma_1,\phi_2,\phi_3,\cdots,\phi_q) \rightarrow 0.
$$
Since the $d$-tuples of endomorphisms above are all in $X$, we conclude that
$$
\chi(K_{\bullet}(\phi,\phi_2,\cdots,\phi_q)) = \chi(K_{\bullet}(\rho_1,\phi_2,\cdots,\phi_q))+\chi(K_{\bullet}(\gamma_1,\phi_2,\cdots,\phi_q)).
$$

Next, we will see that $K_{\bullet}(\phi,\phi_2\cdots,\phi_q)$ is quasi-isomorphic to $K_{\bullet}(\phi_1,\phi_2\cdots,\phi_q)$ as a complex of $R$-modules. Begin with the commutative diagram of $R$-modules with exact rows:

\begin{center}
\begin{tikzcd}
& 0 \arrow[d] & 0 \arrow[d]& 0 \arrow[d]&\\
0   \arrow[r]    & F_1   \arrow[r,"\spmat{ -1_{F_1}\\ \rho_1}"]  \arrow[d,"\phi_1 = \gamma_1\rho_1"']  &  (F_1 \oplus F_1)  \arrow[r,"\spmat{ \rho_1  & 1_{F_1} }"]   \arrow[d,"\spmat{ \rho_1 & 1_{F_1}\\0 & \gamma_1 }"] & F_1  \arrow[d,"1_{F_1}"]  \arrow[r] & 0\\
 0   \arrow[r]    & F_1 \arrow[d]  \arrow[r,"\spmat{ 0\\ 1_{F_1}}"']  &  (F_1 \oplus F_1)  \arrow[r,"\spmat{ 1_{F_1} & 0}"']  \arrow[d] & F_1  \arrow[r] \arrow[d]& 0,\\
 & 0 & 0 & 0 &
\end{tikzcd}
\end{center}
which can be regarded as a short exact sequence of complexes of $R$-modules
$$
0  \rightarrow K_{\bullet}(\phi_1) \rightarrow K_{\bullet}(\phi) \rightarrow K_{\bullet}(1_{F_1}) \rightarrow 0.
$$
Tensoring this with the complex $K_{\bullet}(\phi_2) \otimes \cdots \otimes  K_{\bullet}(\phi_q)$ of free $R$-modules, exactness is preserved and yields the short exact sequence of complexes
$$
0  \rightarrow K_{\bullet}(\phi_1, \phi_2,\cdots,\phi_q) \rightarrow K_{\bullet}(\phi,\phi_2,\cdots,\phi_q) \rightarrow K_{\bullet}(1_{F_1}, \phi_2,\cdots,\phi_q) \rightarrow 0.
$$
Since the last complex is exact, the chain map between the first two is a quasi-isomorphism and therefore
$$
\chi(K_{\bullet}(\phi,\phi_2,\cdots,\phi_q)) = \chi(K_{\bullet}(\phi_1,\phi_2,\cdots,\phi_q)),
$$
finishing the proof.

Step V: Proof of (3): Suppose for simplicity that $k=1$. If $det(\phi_1)$ is a unit then it is easy to see that we may take $a_1=1$.
Set $J = (det(\phi_2),\cdots,det(\phi_q))$. If $J$ is of finite colength in $R$, then by Proposition \ref{prop:trian} (taking the set of primes to be the empty set) we can find an $a_1 \in R$ such that $a_1\phi_1$ is a product of triangularisable endomorphisms.  Further, $(det(a_1\phi_1),det(\phi_2),\cdots,det(\phi_q)) =  (det(a_1\phi_1))+J$ which is of finite colength in $R$, so that $(a_1\phi_1,\phi_2,\cdots,\phi_q) \in X$.

If neither of the previous two cases hold, then $(det(\phi_1),\cdots,det(\phi_q))$ is an ${\mathfrak m}$-primary ideal of $R$ and so $J$ is of height $d-1$.
Take $\{P_1,\cdots,P_n\}$ to be the set of minimal primes of $J$ and apply Proposition \ref{prop:trian}  to get $a_1 \in R$ not in any $P_i$ such that $a_1\phi_1$ is a product of triangularisable endomorphisms. We want that $det(a_1\phi_1)+J = a_1^{r_1}det(\phi_1)+J$ to be of finite colength. But this does hold since both $(a_1)+J$ and  $(det(\phi_1))+J$ are ${\mathfrak m}$-primary.

Step VI: When $R$ is Cohen-Macaulay, we need to see that $K_{\bullet}(\phi_1,\cdots,\phi_d)$ is acyclic. This follows immediately from the Acyclicity Lemma of Peskine-Szpiro, since, the complex has length $d$, the modules in the complex have depth $d$ and the homology modules in the complex are zero or have finite length. (See Lemma 1.8 in \cite{PskSzp1973}). 
 \end{proof}
 
The following theorem  is our main motivation for proving Theorem \ref{thm:comparison}.

\begin{theorem}\label{cor:comparison}
Let $(R,{\mathfrak m},k)$ be a Noetherian local ring of  positive  dimension $d$ with infinite residue field. 
Let $M_1 \subseteq F_1,M_2 \subseteq F_2, \cdots,$ $M_d \subseteq F_d$ be $R$-submodules  of  finite colength in free   $R$-modules $F_1,\cdots,F_d$ of ranks $r_1,\cdots,r_d$.
Let $I_1,\cdots,I_d$ be the ideals of maximal minors of $M_1,\cdots,M_d$. Then
$$
br(M_1|\cdots|M_d) = e(I_1|\cdots|I_d).
$$
\end{theorem}

\begin{proof} Let $(B_1,\cdots,B_d)$ be a joint reduction of $(M_1,\cdots,M_d)$. By Theorem \ref{thm:brmultepchar} the mixed Buchsbaum-Rim multiplicity
$$
br(M_1|\cdots|M_d) = \chi(K_{\bullet}(\phi_1,\phi_2,\cdots,\phi_d))
$$
where $\phi_k$ are endomorphisms of $F_k$ with images $B_k$. By Theorem \ref{thm:comparison}
$$
 \chi(K_{\bullet}(\phi_1,\phi_2,\cdots,\phi_d)) =  \chi(K_{\bullet}(det(\phi_1),det(\phi_2),\cdots,det(\phi_d))).
$$
By Theorem \ref{thm:equiv}, $det(\phi_1),det(\phi_2),\cdots,det(\phi_d)$ form a joint reduction of $I_1,\cdots,I_d$ and so 
$$
e(I_1|\cdots|I_d) =  \chi(K_{\bullet}(det(\phi_1),det(\phi_2),\cdots,det(\phi_d))),
$$
again by Theorem \ref{thm:brmultepchar}. The three equalities together finish the proof.
\end{proof}

\section{The Hoskin-Deligne formula and reduction-number-zero again}

In this section, we give a different proof of Theorem \ref{thm:jtred0} using the Hoskin-Deligne length formula and our results on the joint Buchsbaum-Rim multiplicity. As in \S2, throughout this section, $(R,{\mathfrak m},k)$ will be a two-dimensional regular local ring with infinite residue field.

To motivate the proof for integrally closed modules, we first give a proof for ${\mathfrak m}$-primary integrally closed ideals in the following proposition. Many of the results in the previous two sections were developed precisely in order to generalise this proof.

\begin{proposition}
If $I$ and $J$ are ${\mathfrak m}$-primary integrally closed ideals of $R$ with joint reduction $(a,b)$ then, $IJ = aJ+bI$.
\end{proposition}

\begin{proof}
Step I: There is an natural isomorphism 
$$
\frac{R}{I} \oplus \frac{R}{J} \rightarrow \frac{(a,b)}{aJ+bI}.
$$
The map is given by $(\overline{r},\overline{s}) \mapsto \overline{sa+rb}$, which is clearly well-defined and surjective. To show injectivity, suppose that
$sa+rb \in aJ+bI$. Then $sa+rb = ja+ib$ for some $i \in I$, $j \in J$. Hence, $a(s-j) = b(i-r)$ and since $a,b$ form a regular sequence, there is a $t \in R$ such that $s-j =tb$ and $i-r =ta$. So $s \in J, r \in I$ as needed and hence
$$
\lambda\left( \frac{R}{aJ+bI}\right)  = \lambda \left( \frac{R}{(a,b)}\right) + \lambda \left(\frac{R}{I}\right) + \lambda\left(\frac{R}{J}\right).
$$
Step II: By the Hoskin-Deligne length formula for ideals (see \cite{JhnVrm1992}, Theorem 3.10)
$$
\lambda\left(\frac{R}{I}\right) + \lambda\left(\frac{R}{J}\right) =  \sum_T \left\{ \binom{o_T(I^T)+1}{2} + \binom{o_T(J^T)+1}{2}\right\} [T:R].
$$
Here, the sum is over all quadratic transforms $T$ of $R$, the ideals $I^T$ and $J^T$ are the transforms of $I$ and $J$ in $T$, the notation $o_T(\cdot)$ is the order valuation of $T$ and $[T:R]$ is the degree of the residue field extension of $T$ over $R$.
Similarly
$$
\lambda\left(\frac{R}{IJ}\right) =  \sum_T \binom{o_T(I^T)+o_T(J^T)+1}{2} [T:R]
$$
Step III: Another consequence of the Hoskin-Deligne length formula is the following formula for the mixed multiplicity $e(I|J)$ of $I$ and $J$ (see \cite{JhnVrm1992}, Theorem 3.7)
$$
\lambda \left(\frac{R}{(a,b)}\right) = e(I|J) = \sum_T o_T(I^T)o_T(J^T) [T:R]
$$
Step IV: At every $T$, we have the elementary numerical identity
$$
\binom{o_T(I^T)+o_T(J^T)+1}{2}  = \binom{o_T(I^T)+1}{2} + \binom{o_T(J^T)+1}{2} + o_T(I^T)o_T(J^T).
$$

Putting everything together, $\lambda\left( \frac{R}{aJ+bI}\right) = \lambda\left(\frac{R}{IJ}\right)$ and so $IJ = aJ+bI$.
\end{proof} 

Before proving an analogue of Step I for modules in Proposition \ref{prop:step1analogue}, we need a Tor computation.

\begin{lemma} \label{lem:torcomp} Let $B_1 \subseteq F_1$, $B_2 \subseteq F_2$  all be free $R$-modules with $rank(B_i) = rank(F_i) = r_i$. Suppose that $det(B_1),det(B_2)$ generate an ideal of finite colength  in $R$. Then $Tor_i^R(F_1/B_1,F_2/B_2)$ vanishes for $i \geq 1$.
\end{lemma}

\begin{proof} The tensor product of the two complexes
\begin{center}
\begin{tikzcd}
 0  \arrow[r] & F_1 \arrow[r,"B_1"] & F_1   \arrow[r]  & 0, \\
 0  \arrow[r] & F_2\arrow[r,"B_2"] & F_2   \arrow[r]  & 0
\end{tikzcd}
\end{center}
resolving $\frac{F_1}{B_1}$ and $\frac{F_2}{B_2}$ computes  $Tor_*^R(\frac{F_1}{B_1},\frac{F_2}{B_2})$. Explicitly, this complex is given by

$$
0 \rightarrow F_1 \otimes F_2 \stackrel{\partial_2}{\longrightarrow} (F_1 \otimes F_2) \oplus  (F_1 \otimes F_2) \stackrel{\partial_1}{\longrightarrow} F_1 \otimes F_2 \rightarrow 0,
$$
where the matrices of the maps are given by:
$$
\partial_2 = 
\left[
\begin{array}{r}
I_{r_1} \otimes B_2\\
-B_1 \otimes I_{r_2}
\end{array}
\right]
{\text {~and~}}
\partial_1 = 
\left[
\begin{array}{rr}
B_1 \otimes I_{r_2} & I_{r_1} \otimes B_2\\
\end{array}
\right].
$$
The ideals of maximal minors of both these matrices contain $det(B_1)^{r_2}$ and $det(B_2)^{r_1}$ and so by the Buchsbaum-Eisenbud exactness criterion, this complex resolves the module
$\frac{F_1}{B_1} \otimes \frac{F_2}{B_2} = \frac{F_1F_2}{B_1F_2+F_1B_2}$. In particular, the higher Tors of $\frac{F_1}{B_1}$ and $\frac{F_2}{B_2}$ vanish.
\end{proof}

\begin{proposition} \label{prop:step1analogue} Let $M_1 \subseteq F_1$ of rank $r_1$, $M_2 \subseteq F_2$ of rank $r_2$ be integrally closed of finite non-zero colength with joint reduction $(B_1,B_2)$. 
Then, the  natural map (direct sum of the natural inclusions)
$$
\left(  B_1 \otimes  F_2 \right)  \oplus \left( F_1 \otimes B_2 \right)  \rightarrow F_1 \otimes F_2 = F_1F_2%
$$
has image $B_1F_2+F_1B_2$. The composite map
$$
\left(  B_1 \otimes  F_2 \right)  \oplus \left( F_1 \otimes B_2 \right)  \rightarrow B_1F_2+F_1B_2 \rightarrow \frac{B_1F_2+F_1B_2}{B_1M_2+M_1B_2}.
$$
has kernel $\left(  B_1 \otimes  M_2 \right)  \oplus \left( M_1 \otimes B_2 \right)$. In particular,
$$
r_1 \lambda\left(\frac{F_2}{M_2}\right) + r_2 \lambda\left(\frac{F_1}{M_1}\right) = \lambda\left( \frac{B_1F_2+F_1B_2}{B_1M_2+M_1B_2}\right).
$$

\end{proposition}

\begin{proof} 
Suppose that $F_1$ has basis $X_1,\cdots,X_{r_1}$ and that $F_2$ has basis $Y_1,\cdots,Y_{r_2}$ so that $F_1 \otimes F_2 = F_1F_2$ has basis all the $X_i \otimes Y_j = X_iY_j$. Then $B_1$ has a basis of linear forms, say, $L_1,\cdots,L_{r_1}$ in $X_1,\cdots,X_{r_1}$ and $B_2$ has a basis of linear forms, say, $H_1,\cdots,H_{r_2}$ in $Y_1,\cdots,Y_{r_2}$. The modules $B_1 \otimes F_2$ and $F_1 \otimes B_2$ have bases $L_i \otimes Y_j$ and $X_i \otimes H_j$ respectively, so the image of the natural map, which is clearly $B_1F_2+F_1B_2$, is spanned by all the $L_iY_j$ and $X_iH_j$.

To verify the assertion about the kernel, suppose that 
$$
\sum_{i}  L_i \otimes P_i + \sum_{j} Q_j \otimes H_j
$$
is in the kernel of the composite map where $P_i$ are linear forms in the $Y$'s and $Q_j$ are linear forms in the $X$'s. So there exist elements $C_j \in M_1$ and $D_i \in M_2$ such that
$$
\sum_{i}  L_i  P_i + \sum_{j} Q_j H_j = \sum_{i}  L_i  D_i + \sum_{j} C_j  H_j \Rightarrow \sum_{i}  L_i  (P_i-D_i) + \sum_{j} (Q_j-C_j) H_j = 0.
$$

Hence $\sum_{i}  X_i \otimes  (P_i-D_i) + \sum_{j} (Q_j-C_j) \otimes Y_j \in ker(\partial_2) = im(\partial_1)$, from Lemma~\ref{lem:torcomp}.
Thus $P_i -D_i \in B_2$ and $Q_j-C_j \in B_1$. Hence $P_i \in M_2$ and $Q_j \in M_1$. Note that the hypotheses of Lemma~\ref{lem:torcomp} hold by Theorem~\ref{thm:properties}.
\end{proof}

The analogue of Step II for modules is the Hoskin-Deligne length formula for integrally closed modules which is stated below - see Theorem 3.12 of \cite{KdyMhn2015}.
\begin{theorem} \label{thm:hdf} Let $M \subseteq F$ be an integrally closed $R$-module of finite colength in a free module $F$. Then
$$
\lambda_R \left( \frac{F}{M} \right) = \sum_T \lambda_T \left( \frac{(MT)^{**}}{MV_T \cap (MT)^{**}} \right) [T:R],
$$
where the sum extends over all quadratic transforms $T$ of $R$, the ring $V_T$ is the order valuation ring associated to $T$ and $MT$ is the transform of $M$ in $T$.
\end{theorem}
This theorem expresses the colength of an integrally closed module in terms of data associated to local submodules for quadratic transforms of $R$, where, by `local', we mean `contracted from the order valuation'.

The analogue of the first equality of Step III for modules is the following equality:
$$
\lambda\left( \frac{F_1F_2}{B_1F_2+F_1B_2} \right) = br(M_1|M_2),
$$
which follows from Theorem \ref{thm:brmultepchar} and Theorem \ref{thm:comparison} since $R$ is Cohen-Macaulay. The analogue of the second equality of Step III for modules is the following equality:
$$
br(M_1|M_2) =\sum_T o_T(I_1^T)o_T(I_2^T)  [T:R],
$$
which follows from Corollary \ref{cor:comparison} and the Hoskin-Deligne formula for ideals.

The analogue of Step IV for modules is contained  in Theorem \ref{thm:local} before proving which we need some preparatory results about local modules.

\begin{proposition}\label{prop:prodlocal}
If $M_1,M_2$ are local submodules for $R$, then so is $M_1M_2$.
\end{proposition}

\begin{proof}
A characterisation of local submodules - see Theorem 3.7(4) of \cite{KdyMhn2015}  - is that they be integrally closed with ideal of maximal minors a power of ${\mathfrak m}$.
If both $M_1,M_2$ satisfy these conditions, then $M_1M_2$ is integrally closed and by Theorem~1(1) of \cite{BrnVsc2003} the ideal
$$
I(M_1M_2) = I(M_1)^{r_2}I(M_2)^{r_1}
$$
since all of these are integrally closed.  Thus $I(M_1M_2)$ is a power of ${\mathfrak m}$ and so $M_1M_2$ is also local.
\end{proof}

\begin{corollary} \label{cor:localprod}
For any submodules $M_1 \subseteq F_1$ and  $M_2 \subseteq F_2$ of finite colength,
$$
M_1M_2V \cap F_1F_2 = (M_1V \cap F_1)(M_2V \cap F_2).
$$
\end{corollary}

\begin{proof} First, for any $M \subseteq F$, $(MV \cap F)V = MV$. Applying this to $M_1$ and $M_2$, we see that $$ M_1M_2V = M_1V M_2V = (M_1V \cap F_1)V(M_2V \cap F_2)V = (M_1V \cap F_1)(M_2V \cap F_2)V.$$
Now intersect with $F_1F_2$ and use Proposition \ref{prop:prodlocal} to conclude that
$$M_1M_2V \cap F_1F_2 = (M_1V \cap F_1)(M_2V \cap F_2),$$
since $M_1V \cap F_1$ and $M_2V \cap F_2$ are local.
\end{proof}

 \begin{theorem}\label{thm:local}  If $M_1 \subseteq F_1$ and $M_2\subseteq F_2$ are local submodules for $R$ of ranks $r_1, r_2$ and orders $n_1, n_2$ respectively, then,
$$
\lambda\left( \frac{F_1F_2}{M_1M_2} \right) = r_2 \lambda\left( \frac{F_1}{M_1} \right)  + r_1 \lambda\left( \frac{F_2}{M_2} \right) +
n_1n_2
$$
\end{theorem}

We will first show that the inequality $\leq$ always holds in the equation above.

\begin{lemma}\label{lem:ineq}
With notation as in Theorem \ref{thm:local}, 
$$
\lambda\left( \frac{F_1F_2}{M_1M_2} \right) \leq r_1 \lambda\left( \frac{F_2}{M_2} \right) + r_2 \lambda\left( \frac{F_1}{M_1} \right) + n_1n_2.
$$
\end{lemma}

\begin{proof}
To see this, let $(B_1,B_2)$ be a joint reduction of $(M_1,M_2)$. Then $B_1,B_2$ are free submodules of $F_1,F_2$ of ranks $r_1,r_2$ by Theorem \ref{thm:properties}. By Proposition \ref{prop:step1analogue},
$$
r_1 \lambda\left( \frac{F_2}{M_2} \right) + r_2 \lambda\left( \frac{F_1}{M_1} \right) = \lambda\left( \frac{B_1F_2+F_1B_2}{B_1M_2+M_1B_2} \right).
$$

Since $M_1,M_2$ are local submodules, $I_1 = {\mathfrak m}^{n_1}, I_2 = {\mathfrak m}^{n_2}$ and so 
by Corollary 19 and Theorem \ref{thm:brmultepchar}, 
$$
n_1n_2 = e(I_1|I_2) = br(M_1|M_2) = \lambda\left( \frac{F_1F_2}{B_1F_2+F_1B_2} \right).
$$
Thus the right hand side of the desired inequality evaluates to $\lambda\left( \frac{F_1F_2}{B_1M_2+M_1B_2} \right)$ and 
since $B_1M_2+M_1B_2 \subseteq M_1M_2$, we're done.
\end{proof}

\begin{proof}[Proof of Theorem \ref{thm:local}] We first demonstrate the case $q = 2$. For this, we induce on $r_2$. When $r_2=1$, the module $M_2={\mathfrak m}^t$ and $F_2=R$ for some $t \geq 0$. So we need to see that
$$
\lambda\left( \frac{F_1}{{\mathfrak m}^{n_2}M_1} \right) = r_1 \lambda\left( \frac{R}{{\mathfrak m}^{n_2}} \right) +  \lambda\left( \frac{F_1}{M_1} \right) + n_1n_2.
$$
The left hand side of this equation equals
\begin{eqnarray*}
  &\lambda\left( \frac{F_1}{M_1} \right) + \mu(M_1) + \mu({\mathfrak m}M_1) + \cdots + \mu({\mathfrak m}^{n_2-1}M_1) =\\
  &\lambda\left( \frac{F_1}{M_1} \right) + (n_1+r_1) + (n_1+2r_1) + \cdots + (n_1+n_2r_1) =\\
  &\lambda\left( \frac{F_1}{M_1} \right) + n_1n_2+r_1\binom{n_2+1}{2}
\end{eqnarray*}
which agrees with the right hand side, proving equality in the basis case.

When $r_2 >1$, we may assume that $M_2$ has no direct summand by induction since both sides are additive with respect to direct sum decompositions of $M_2$. 
Now we will induce on $\lambda\left( \frac{F_2}{M_2} \right)$. If this is at most $r_2$ or more generally if $M_2 \not\subseteq {\mathfrak m}F_2$, then $M_2$ necessarily has a free direct summand, so we're done.  So we assume that $M_2 \subseteq {\mathfrak m}F_2$. By the inductive hypothesis the result holds for $(M_1,M_2:{\mathfrak m})$ provided $M_2:{\mathfrak m}$ is contracted from the order valuation of $R$, which we next verify.
Suppose that $v \in F_2 \cap (M_2:{\mathfrak m})V$. 
Consider ${\mathfrak m}v \subseteq {\mathfrak m}(M_2:{\mathfrak m})V \cap F_2 \subseteq M_2V \cap F_2 = M_2.$ So $v \in M_2:{\mathfrak m}$, as desired.

It is easy to verify that if the proposition holds for $(M_1,M_2)$ then it also does for $(M_1,{\mathfrak m}M_2)$ - note that ${\mathfrak m}M_2$ is contracted from the order valuation of $R$ by Theorem~3.7 of \cite{KdyMhn2015}. So the proposition holds for $(M_1,M_2^\prime = {\mathfrak m}(M_2:{\mathfrak m}))$, giving the equation
$$
\lambda\left( \frac{F_1F_2}{M_1M_2^\prime} \right) = r_1 \lambda\left( \frac{F_2}{M_2^\prime} \right) + r_2 \lambda\left( \frac{F_1}{M_1} \right) + n_1n_2^\prime
$$
where $n_2^\prime = ord(M_2^\prime)$.

Hence to prove that the proposition holds for $(M_1,M_2)$ it suffices to verify the equation
$$
\lambda\left( \frac{M_1M_2}{M_1M_2^\prime} \right) = r_1 \lambda\left( \frac{M_2}{M_2^\prime} \right)  + n_1(n_2^\prime - n_2).
$$
Actually by Lemma \ref{lem:ineq}, it suffices to verify the inequality
$$
\lambda\left( \frac{M_1M_2}{M_1M_2^\prime} \right) \leq r_1 \lambda\left( \frac{M_2}{M_2^\prime} \right)  + n_1(n_2^\prime - n_2).
$$

We will now see that 
$$
\lambda\left( \frac{M_2}{M_2^\prime} \right) = n_2^\prime - n_2.
$$
This follows from considering the chain
$$
M_2^\prime = {\mathfrak m}(M_2:{\mathfrak m}) \subseteq M_2 \subseteq M_2:{\mathfrak m}
$$
and observing that 
the length of the whole chain is 
$$
\mu(M_2:{\mathfrak m}) = ord(M_2:{\mathfrak m}) + r_2 = ord(M_2^\prime) = n_2^\prime
$$
 while the length of the top link is $dim ~Tor_1(M_2,k) = ord(M_2) = n_2$.

So we're down to checking that
$$
\lambda\left( \frac{M_1M_2}{M_1M_2^\prime} \right) \leq (r_1 + n_1)\lambda\left( \frac{M_2}{M_2^\prime} \right) = \mu(M_1)\lambda\left( \frac{M_2}{M_2^\prime} \right).
$$

To see this note that $\frac{M_1M_2}{M_1M_2^\prime}$ and $\frac{M_2}{M_2^\prime}$ are both vector spaces over $k$, so the lengths are just the dimensions. It suffices to see that $\frac{M_1M_2}{M_1M_2^\prime}$ has a spanning set with at most $\mu(M_1)\lambda\left( \frac{M_2}{M_2^\prime} \right)$ elements.

We have $\frac{M_2^\prime}{{\mathfrak m}M_2} \subseteq \frac{M_2}{{\mathfrak m}M_2}$ with codimension $dim\left( \frac{M_2}{M_2^\prime}\right)$. Choosing a splitting and lifting bases back to $M_2$ results in elements $\{H_1,\cdots,H_t\}$ in $M_2^\prime$ whose images in $\frac{M_2^\prime}{{\mathfrak m}M_2}$ form a basis and elements $\{H_{t+1},\cdots,H_{t+u}\}$ in $M_2$ which, together with $\{H_1,\cdots,H_t\}$, form a minimal generating set of $M_2$. Thus $\lambda\left( \frac{M_2}{M_2^\prime} \right) = u$.

Suppose that $L_1,\cdots,L_{r_1+n_1}$ are a minimal generating set of $M_1$. Then all the $L_iH_j \in F_1F_2$ generate $M_1M_2$. If $j \leq t$ then $L_iH_j \in M_1M_2^\prime$, so the quotient  $\frac{M_1M_2}{M_1M_2^\prime}$ is spanned by all the $L_iH_j$ with $t+1 \leq j \leq t+u$ - a total of $(n_1+r_1)u$ elements, as needed.

\medskip
As before, the case $q > 2$ reduces to the case $q = 2$ using the fact that a product of integrally closed modules is integrally closed and that the order function is multiplicative.
\end{proof}

We now put the preceding results together to give our second proof of Theorem~\ref{thm:jtred0}. 

\begin{proof}[Proof of Theorem \ref{thm:jtred0}] We will first show, generalising Theorem \ref{thm:local}, that for integrally closed modules, $M_1,M_2$,
$$
\lambda\left( \frac{F_1F_2}{M_1M_2} \right) = r_1 \lambda\left( \frac{F_2}{M_2} \right) + r_2 \lambda\left( \frac{F_1}{M_1} \right) + e(I_1|I_2).
$$
To see this, apply the Hoskin-Deligne length formula for integrally closed modules - Theorem \ref{thm:hdf} - to the three length terms and its consequence for the mixed multiplicity to the last term to see that it suffices to show that for every quadratic transform $T$ of $R$,
\begin{align*}
\lambda_T \left( \frac{(M_1M_2T)^{**}}{M_1M_2V_T \cap (M_1M_2T)^{**}} \right) =  r_1\lambda_T \left( \frac{(M_2T)^{**}}{M_2V_T \cap (M_2T)^{**}} \right)\\
 + r_2\lambda_T \left( \frac{(M_1T)^{**}}{M_1V_T \cap (M_1T)^{**}} \right) + o_T(I_1^T)o_T(I_2^T).
\end{align*}
Using Corollary \ref{cor:localprod}, the term on the left of the equation evaluates to
$$
\lambda_T \left( \frac{(M_1T)^{**}(M_2T)^{**}}{(M_1V_T \cap (M_1T)^{**})(M_2V_T \cap (M_2T)^{**})} \right).
$$
Now use the fact that $I(MV \cap F) = I(M)V \cap F$ applied to $M_1^T$ and $M_2^T$ and Theorem \ref{thm:local} applied to $M_1V^T \cap (M_1T)^{**}$ and $M_2V^T \cap (M_2T)^{**}$ to conclude that the desired equality does hold.

However, we also have that
\begin{eqnarray*}
\lambda\left( \frac{F_1F_2}{B_1M_2+M_1B_2} \right) &=& \lambda\left( \frac{B_1F_2+F_1B_2}{B_1M_2+M_1B_2} \right)  + \lambda\left( \frac{F_1F_2}{B_1F_2+F_1B_2} \right)\\
&=& r_1 \lambda\left( \frac{F_2}{M_2} \right) + r_2 \lambda\left( \frac{F_1}{M_1} \right) + e(I_1|I_2)
\end{eqnarray*}
where the second equality follows from Proposition \ref{prop:step1analogue},  Corollary \ref{cor:comparison}, Theorem \ref{thm:comparison} and Theorem \ref{thm:brmultepchar}.

Finally, since
$$
\lambda\left( \frac{F_1F_2}{B_1M_2+M_1B_2} \right) = \lambda\left( \frac{F_1F_2}{M_1M_2} \right),
$$
it follows that $M_1M_2 = B_1M_2+M_1B_2$.
\end{proof}

\section{The joint Buchsbaum-Rim function for integrally closed modules over two-dimensional regular local rings} Throughout this section $(R,{\mathfrak m},k)$ will be a two-dimensional regular local ring with infinite residue field. Our goal in this section is to give an explicit expression for the joint Buchsbaum-Rim function of several integrally closed modules over $R$. The case of a single integrally closed module was given in \cite{KtzKdy1997}. We begin with a proposition that gives the colength of a product of integrally closed modules.

\begin{proposition}\label{prop:prodlength} Suppose that $M_1 \subseteq F_1$, $\cdots$, $M_q \subseteq F_q$ (for $q \geq 2$) are integrally closed modules of finite colength in free modules $F_k$ of rank $r_k$. Let $I_k$ denote the ideal of maximal minors of $M_k$. For $1 \leq i < j \leq q$, set
\begin{eqnarray*}
s_i &=& r_1 \cdots \widehat{r_i} \cdots r_q,\\
t_{ij} &=& r_1 \cdots \widehat{r_i} \cdots \widehat{r_j} \cdots r_q.
\end{eqnarray*}
Then,
$$
\lambda\left( \frac{F_1\cdots F_q}{M_1\cdots M_q} \right) = \sum_{i=1}^q s_i \lambda\left( \frac{F_i}{M_i} \right) + \sum_{1 \leq i < j \leq q}   t_{ij}e(I_i|I_j).
$$
\end{proposition}

\begin{proof} The proof is by induction on $q \geq 2$ where the basis case $q=2$ is proved in the  first equation in the proof of Theorem~\ref{thm:jtred0} of the last section.
For $q > 2$ we first introduce the following notation for $2 \leq i<j \leq q$.
\begin{eqnarray*}
\tilde{s}_i &=& r_2 \cdots \widehat{r_i} \cdots r_q = t_{1i} = \frac{s_i}{r_1},\\
\tilde{t}_{ij} &=& r_2 \cdots \widehat{r_i} \cdots \widehat{r_j} \cdots r_q = \frac{t_{ij}}{r_1}.
\end{eqnarray*}
We then see that
\begin{eqnarray*}
\lambda\left( \frac{F_1\cdots F_q}{M_1\cdots M_q} \right) &=& r_1 \lambda\left( \frac{F_2\cdots F_q}{M_2\cdots M_q} \right) + s_1 \lambda\left( \frac{F_1}{M_1} \right) + e(I_1|I(M_2\cdots M_q))\\
&=& r_1\left\{ \sum_{i=2}^q \tilde{s}_i \lambda\left( \frac{F_i}{M_i} \right) + \sum_{2 \leq i < j \leq q}   \tilde{t}_{ij}e(I_i|I_j) \right\} + s_1 \lambda\left( \frac{F_1}{M_1} \right) \\ 
& &  + e(I_1|I_2^{\tilde{s}_2}\cdots I_q^{\tilde{s}_q}))\\
&=&  \sum_{i=2}^q {s}_i \lambda\left( \frac{F_i}{M_i} \right) + \sum_{2 \leq i < j \leq q}   {t}_{ij}e(I_i|I_j) + s_1 \lambda\left( \frac{F_1}{M_1} \right) \\
& & + \sum_{i=2}^q \tilde{s}_i e(I_1|I_i) \\
&=& \sum_{i=1}^q s_i \lambda\left( \frac{F_i}{M_i} \right) + \sum_{1 \leq i < j \leq q}   t_{ij}e(I_i|I_j),
\end{eqnarray*}
as needed.
The first equality follows from the $q=2$ case (and the fact that $M_2\cdots M_q$ is integrally closed). The second follows from the inductive assumption and that $I(M_2\cdots M_q) = I_2^{\tilde{s}_2}\cdots I_q^{\tilde{s}_q}$ (which itself follows inductively from the assertion in the $q=2$ case as in Proposition \ref{prop:prodlocal}). 
The third equality follows from $s_i = r_1\tilde{s}_i$ and the properties of the mixed multiplicity symbol $e(\cdot|\cdot)$, while the last equality uses $\tilde{s}_i = t_{1i}$.
\end{proof}

A straightforward application of Proposition \ref{prop:prodlength} yields an explicit expression for the joint Buchsbaum-Rim function of several integrally closed modules over $R$.
\begin{theorem}\label{thm:brpolya} Suppose that $M_1 \subseteq F_1$, $\cdots$, $M_q \subseteq F_q$ (for $q \geq 2$) are integrally closed modules of finite colength in free modules $F_k$ of rank $r_k$. The joint Buchsbaum-Rim function of $M_1,\cdots,M_q$ is given by 
\begin{eqnarray*}
&\lambda&\!\!\!\!\!\!\left( \frac{S_{n_1}(F_1)\cdots S_{n_q}(F_q)}{S_{n_1}(M_1)\cdots S_{n_q}(M_q)} \right)\\ &=& \sum_{i=1}^q s_i \left[ br(M_i) \binom{n_i+r_i}{r_i+1} - \left( br(M_i) -\lambda(\frac{F_i}{M_i}) \right) \binom{n_i+r_i-1}{r_i} \right] \\
&+& \sum_{1 \leq i < j \leq q}   v_{ij}e(I_i|I_j)
\end{eqnarray*}
for all $n_1,\cdots,n_q \geq 0$, where $s_i$ and $v_{ij}$ are given by
\begin{eqnarray*}
s_i &=& \binom{n_1+r_1-1}{r_1-1} \cdots \widehat{\binom{n_i+r_i-1}{r_i-1}} \cdots \binom{n_q+r_q-1}{r_q-1},\\
v_{ij} &=& \binom{n_1+r_1-1}{r_1-1} \cdots {\binom{n_i+r_i-1}{r_i}} \cdots {\binom{n_j+r_j-1}{r_j}} \cdots \binom{n_q+r_q-1}{r_q-1}.
\end{eqnarray*}
\end{theorem}

\begin{proof} By Proposition \ref{prop:prodlength} (and since products and torsion-free symmetric powers of integrally closed modules are integrally closed),
\begin{align*}
\lambda\left( \frac{S_{n_1}(F_1)\cdots S_{n_q}(F_q)}{S_{n_1}(M_1)\cdots S_{n_q}(M_q)} \right) = \sum_{i=1}^q s_i \lambda\left( \frac{S_{n_i}(F_i)}{S_{n_i}(M_i)} \right) \\
+  \sum_{1 \leq i < j \leq q}   t_{ij}e(I(S_{n_i}(M_i)|S_{n_j}(M_j)),
\end{align*}
where $s_i$ is the product of the ranks of all the $S_{n_k}(F_k)$ except $S_{n_i}(F_i)$ and $t_{ij}$ is the product of ranks of all the $S_{n_k}(F_k)$ except $S_{n_i}(F_i)$ and $S_{n_j}(F_j)$. 

The rank of $S_{n_k}(F_k)$ is given by $\binom{n_k+r_k-1}{r_k-1}$ while $I(S_{n_k}(M_k)) = I_k^{\binom{n_k+r_k-1}{r_k}}$, where the second assertion follows from Theorem 1(1) of \cite{BrnVsc2003}. Thus,
\begin{eqnarray*}
s_i &=& \binom{n_1+r_1-1}{r_1-1} \cdots \widehat{\binom{n_i+r_i-1}{r_i-1}} \cdots \binom{n_q+r_q-1}{r_q-1},\\
t_{ij} &=& \binom{n_1+r_1-1}{r_1-1} \cdots \widehat{\binom{n_i+r_i-1}{r_i-1}} \cdots \widehat{\binom{n_j+r_j-1}{r_j-1}} \cdots \binom{n_q+r_q-1}{r_q-1}.
\end{eqnarray*}
The term $t_{ij}e(I(S_{n_i}(M_i)|S_{n_j}(M_j))$ simplifies to $v_{ij}e(I_i|I_j)$.

Finally, by Proposition 3.4 of \cite{KtzKdy1997} the length of $\frac{S_{n_i}(F_i)}{S_{n_i}(M_i)}$ is given by
$$
br(M_i) \binom{n_i+r_i}{r_i+1} - \left( br(M_i) -\lambda(\frac{F_i}{M_i}) \right) \binom{n_i+r_i-1}{r_i}.
$$
Putting everything together completes the proof.
\end{proof}

\section{Appendix: The joint Buchsbaum-Rim polynomial}

In this section we give a self contained proof of  the following theorem. This 
theorem
is subsumed by considerably more general statements in various papers (e.g., \cite{KrbRes1994} and  \cite{KlmThr1996}) but is difficult to extract in the form we are using. Thus, a brief exposition seems worthwhile. 

\begin{theorem} \label{thm:brpoly}
Let $(R,{\mathfrak m},k)$ be a Noetherian local ring of positive dimension $d$. For $q \in {\mathbb N}$ and $k=1,2,\cdots,q$, let $M_k \subseteq F_k$ be modules of finite non-zero colength with $F_k$ free of rank $r_k$. Then, the function
$$
f(n_1,\cdots,n_q) =  \lambda\left(
\frac
{S_{n_1}(F_1)\cdots S_{n_q}(F_q)}
{S_{n_1}(M_1)\cdots S_{n_q}(M_q)}
\right)
$$
agrees with  a polynomial function $p(n_1,\cdots,n_q)$ if all of $n_1,\cdots,n_q$ are sufficiently large. The total degree of $p(n_1,\cdots,n_q)$  is $d+r_1+\cdots+r_q-q$.
\end{theorem}

\begin{proof}
The proof is by induction on $q$. Contemporary proofs of the base case $q = 1$ can be found in \cite{Rbr1998} or \cite{SmsLrcVsc2001}, but we present a proof of this case, for the sake of completeness, and to show how the general case follows naturally from the base case. 

For the base case, we start with $M\subseteq F$, so that $\frac{F}{M}$ has finite length and $F$ has rank $r$. Set ${\mathcal S}_0 = S(F)$. Let $t$ be an indeterminate over ${\mathcal S}_0$ having degree zero, and set ${\mathcal S} = {\mathcal S}_0[t]$ and ${\mathcal R} = {\mathcal S}_0[Mt] = R[F, Mt]$ so that ${\mathcal R}$ is a Noetherian, standard $\N$-graded $R$-subalgebra of ${\mathcal S}$. Note, that as an $R$-module, ${\mathcal R}_n = (F, Mt)^n = F^n \oplus F^{n-1}Mt \oplus \cdots \oplus M^nt^n$. Here, and in the rest of the proof, we will use notation such as $F^n$ and $M^n$ to denote $S_n(F)$ and $S_n(M)$.

If we consider the graded ${\mathcal R}$-module $\frac{F{\mathcal R}}{M{\mathcal R}}$, which is a quotient of ideals of ${\mathcal R}$, we see that, as an $R$-module,
\[\left(\frac{F{\mathcal R}}{M{\mathcal R}}\right)_n = \frac{F^n}{F^{n-1}M} \oplus \frac{F^{n-1}M}{F^{n-2}M^2} \oplus \cdots \oplus \frac{FM^{n-1}}{M^n},
\] so that the length of $\left(\frac{F{\mathcal R}}{M{\mathcal R}}\right)_n$ equals the length of $\frac{F^n}{M^n}$. Since the graded components of the finite ${\mathcal R}$-module $\frac{F{\mathcal R}}{M{\mathcal R}}$ have finite length, the length of $\frac{F^n}{M^n}$ agrees with  a polynomial in $n$, for all $n$ sufficiently large. 

We now proceed with the induction argument, assuming the result for $q-1$. Towards that end, we will further simplify notation as follows: We write $\n \in {\mathbb N}^{q-1}$ for the $(q-1)$-tuple $(n_1, \ldots, n_{q-1})$ and interpret ``$n$ sufficiently large"  to mean that  each  coordinate $n_k$ is sufficiently large. We set $F^{\n} = F_1^{n_1}\cdots F_{q-1}^{n_{q-1}}$ and $M^{\n} = 
M_1^{n_1}\cdots M_{q-1}^{n_{q-1}}$. We will show that the lengths of $\frac{F^nF_q^{n_q}}{M^nF_q^{n_q}}$ and $\frac{M^nF_q^{n_q}}{M^nM_q^{n_q}}$ are given by polynomials in $n_1, \ldots, n_q$ for $n$ and $n_q$ sufficiently large. Adding the resulting polynomials will show that the length of $\frac{F_1^{n_1}\cdots F_{q}^{n_{q}}}{M_1^{n_1}\cdots M_{q}^{n_{q}}}$ is given by a polynomial, say $p(n_1, \ldots, n_q)$, when all $n_i$ are sufficiently large.

To see that the length of $\frac{F^nF_q^{n_q}}{M^nF_q^{n_q}}$  is given by a polynomial, consider the ${\mathbb N}^q$-graded $R$-algebra ${\mathcal S}_0 = S(F)$ where $F = F_1 \oplus \cdots \oplus F_q$. It is easy to see that $F^nF_q^{n_q}$ is a direct sum of $rk(F_q^{n_q}) = \binom{n_q+r_q-1}{r_q-1}$ copies of $F^n$, while $M^nF_q^{n_q}$
is a direct sum of $\binom{n_q+r_q-1}{r_q-1}$ copies of $M^n$ sitting in $F^nF_q^{n_q}$ in the obvious manner, so that the quotient $\frac{F^nF_q^{n_q}}{M^nF_q^{n_q}}$ is isomorphic to a direct sum of $\binom{n_q+r_q-1}{r_q-1}$ copies of $\frac{F^n}{M^n}$. By induction on $q$, the length of $\frac{F^n}{M^n}$ is given by a polynomial $P(n)$ for $n$ sufficiently large. It follows that the length of $\frac{F^nF_q^{n_q}}{M^nF_q^{n_q}}$ is given by the polynomial $P(n)\binom{n_q+r_q-1}{r_q-1}$ in $n_1,\cdots, n_q$ when these are all sufficiently large.

To see that the length of $M^nF_q^{n_q}/M^nM_q^{n_q}$ is given by a polynomial, we work inside of the ring ${\mathcal S} = {\mathcal S}_0[t]$, where $t$ is an indeterminate having degree zero. Set 
${\mathcal R} = R[M_1, \ldots, M_{q-1}, F_q, M_qt]$, so that ${\mathcal R}$ is a standard $\N^q$-graded subalgebra of ${\mathcal S}$. It follows that ${\mathcal R}_{(n,n_q)}$ is generated by the degree $(n_1, \ldots, n_q)$ products coming from the set
$M^n\cdot (F_q^{n_q}, F_q^{n_q-1}M_qt, F_q^{n_q-2}M_q^2t^2, \ldots, M_q^{n_q}t^{n_q})$. We now consider the $\N ^q$-graded ${\mathcal R}$-module $\frac{F_q{\mathcal R}}{M_q{\mathcal R}}$. 

Note that $(F_q{\mathcal R})_{(n,n_q)} = F_q{\mathcal R}_{(n,n_q-1)}$ and $(M_q{\mathcal R})_{(n,n_q)} = 
M_q{\mathcal R}_{(n,n_q-1)}$, so from our description of ${\mathcal R}_{(n,n_q)}$, we have that, as an $R$-module
\[
(F_q{\mathcal R})_{(n,n_q)} = M^nF_q^{n_q} \oplus M^nF_q^{n_q-1}M_qt \oplus \cdots \oplus M^nF_qM_q^{n_q-1}t^{n_q-1},
\] and similarly, we have
\[
(M_q{\mathcal R})_{(n,n_q)} = M^nF_q^{n_q-1}M_q \oplus M^nF_q^{n_q-2}M_q^2t \oplus \cdots \oplus M^nM_q^{n_q}t^{n_q-1} .
\] It follows that, as an $R$-module, 
\[
\left(\frac{F_q{\mathcal R}}{M_q{\mathcal R}}\right)_{(n,n_q)} = \frac{M^nF_q^{n_q}}{M^nF_q^{n_q-1}M_s} \oplus \frac{M^nF_q^{n_q-1}M_q}{M^nF_q^{n_q-2}M^2_q} \oplus\cdots \oplus \frac{M^nF_qM_q^{n_q-1}}{M^nM_q^{n_q}}.
\]
Therefore, the length of $\left(\frac{F_q{\mathcal R}}{M_q{\mathcal R}}\right)_{(n,n_q)}$ equals the length of $\frac{M^nF_q^{n_q}}{M^nM_q^{n_q}}$. Since $\frac{F_q{\mathcal R}}{M_q{\mathcal R}}$ is a finitely-generated $\N ^q$-graded module over ${\mathcal R}$, it follows that the length of 
$\frac{M^nF_q^{n_q}}{M^nM_q^{n_q}}$  is given by a polynomial in $n, n_q$ for all $n_i$ sufficiently large.

What about the degree of the polynomial $p(n_1,\cdots,n_q)$? If we assume that each $M_k\subseteq \m F_k$ then $M_1^{n_1}\cdots M_{q}^{n_{q}} \subseteq \m^{n_1+\cdots +n_q}F_1^{n_1}\cdots F_{q}^{n_{q}}$ so that the length of $\frac{F_1^{n_1}\cdots F_{q}^{n_{q}}}{M_1^{n_1}\cdots M_{q}^{n_{q}}}$ is bounded below by 
$\lambda (R/\m^{n_1+\cdots + n_q})\cdot \binom{n_1+r_1-1}{r_1-1}\cdots \binom{n_q+r_q-1}{r_q-1}$, which is ultimately a polynomial of total degree $d+r_1+\cdots + r_q - q$. On the other hand, if $J$ is an $\m$-primary ideal annihilating each $F_k/M_k$, then $J^{n_1+\cdots + n_q}F_1^{n_1}\cdots F_{q}^{n_{q}} \subseteq M_1^{n_1}\cdots M_{q}^{n_{q}}$ so that the length of $\frac{F_1^{n_1}\cdots F_{q}^{n_{q}}}{M_1^{n_1}\cdots M_{q}^{n_{q}}}$ is bounded above by 
$\lambda (R/J^{n_1+\cdots + n_q})\cdot \binom{n_1+r_1-1}{r_1-1}\cdots \binom{n_q+r_q-1}{r_q-1}$, which is ultimately a polynomial also of total degree $d+r_1+\cdots + r_q - q$. Therefore, the degree of $p(n_1,\cdots,n_q)$ is  $d+r_1+\cdots +r_q-q$.

What if some $M_k\not\subseteq \m F_k$? We  will show that the degree of the multivariable Buchsbaum-Rim polynomial is still equal to $d+r_1+\cdots + r_q-q$, as long as each $\frac{F_k}{M_k}$ is not zero. We do this by induction on $r_1+\cdots+ r_q$. If this sum equals one, then we are just in the case of $R$ modulo an $\m$-primary ideal $I$, and in this case the Hilbert-Samuel polynomial of $I$ has degree $d$, which is what we want. If every $M_k\subseteq \m F_k$, then the proof above applies. Without loss of generality, suppose $M_q\not \subseteq \m F_q$. Then, we can write 
$F_q = G_q\oplus C$ and $M = N_q\oplus C$, for $N_q\subseteq G_q$ and $C \subseteq F_q$ a rank one free $R$-module. Then, working inside of $S_0$, for $n_q\geq 1$ we have
\[
F_q^{n_q} = G_q^{n_q}\oplus G_q^{n_q-1}C\oplus \cdots \oplus G_qC^{n_q-1}\oplus C^{n_q},
\] and
\[
M_q^{n_q} = N_q^{n_q}\oplus N_q^{n_q-1}C\oplus \cdots \oplus N_qC^{n_q-1}\oplus C^{n_q}.
\] Maintaining the notation from above, it follows that 

\[
\frac{F^nF_q^{n_q}}{M^nM_q^{n_q}} \cong \frac{M^nG_q^{n_q}}{M^nN_q^{n_q}}\oplus \frac{M^nG_q^{n_q-1}}{M^nN_{q}^{n_q-1}}\oplus \cdots \oplus \frac{G_q}{N_q}.
\]
If we let $H_F(n_1, \ldots, n_q)$ denote the function giving the lengths of  $\frac{F_1^{n_1}F_2^{n_2}\cdots F_q^{n_q}}{M_1^{n_1}M_2^{n_2}\cdots 
M_q^{n_q}}$ and $H_G(n_1, \ldots, n_q)$ for the function giving the lengths of $\frac{F_1^{n_1}F_2^{n_2}\cdots G_q^{n_q}}{M_1^{n_1}M_2^{n_2}\cdots 
N_q^{n_q}}$, then 
\[
H_F(n_1, \ldots, n_s) = \sum _{i = 1}^{n_s} H_G(n_1, \ldots, n_{s-1}, i).
\] It follows that the total degree of the polynomial corresponding to $H_F(n_1, \ldots, n_s)$ is one more than the total degree of the polynomial corresponding to $H_G(n_1, \ldots, n_s)$. Invoking the induction hypothesis completes the proof. 
\end{proof}

\end{document}